\documentclass[12pt,a4]{amsart}
\usepackage{inputenc,amsxtra}
\usepackage{a4wide}
\usepackage[T1]{fontenc}
\usepackage{fixltx2e, graphicx, mathrsfs,longtable, float, wrapfig, soul,%cleveref,
  textcomp, marvosym, wasysym, latexsym, bbm,
  fancybox,fancyvrb, cite,amsmath,amssymb,amsthm,amsrefs,amsfonts,subfig,pdfpages,enumerate}
\usepackage{enumitem}
\usepackage{listings}
\usepackage{tabularx}
\usepackage{pdfpages}
\usepackage{mathtools,hyperref}
 \usepackage{relsize}
\usepackage{blkarray, bigstrut}
\usepackage[all]{xy}
\newtheorem{theorem}{Theorem}[section]
\newtheorem{lemma}[theorem]{Lemma}
\newtheorem{proposition}[theorem]{Proposition}
\newtheorem{corollary}[theorem]{Corollary}

\theoremstyle{definition}
\newtheorem{example}[theorem]{Example}
\newtheorem{definition}[theorem]{Definition}
\newtheorem{remark}[theorem]{Remark}

\newcommand\xqed[1]{%
  \leavevmode\unskip\penalty9999 \hbox{}\nobreak\hfill
  \quad\hbox{#1}}
\newcommand\demo{\xqed{$\triangle$}}

\newcommand{\ignore}[1]{}

\newcommand{\minority}{\operatorname{minority}} 
\newcommand{\majority}{\operatorname{majority}} 

\DeclareMathOperator{\Csp}{CSP} 
\DeclareMathOperator{\Age}{Age} 
\DeclareMathOperator{\Th}{Th} 
\DeclareMathOperator{\Pol}{Pol} 
\DeclareMathOperator{\End}{End} 
\DeclareMathOperator{\Aut}{Aut} 
\newcommand{\bA}{\mathfrak A}
\newcommand{\bU}{\mathfrak U}
\newcommand{\bB}{\mathfrak B}
\newcommand{\bC}{\mathfrak C}
\newcommand{\bD}{\mathfrak D}
\newcommand{\bG}{\mathfrak G}
\newcommand{\bH}{\mathfrak H}
\newcommand{\bK}{\mathfrak K}
\newcommand{\bL}{\mathfrak L}
\newcommand{\bM}{\mathfrak M}
\newcommand{\bP}{\mathfrak P}
\newcommand{\bS}{\mathfrak S}
\newcommand{\bT}{\mathfrak T}

\newcommand{\Proj}{\mathscr P}

\renewcommand{\Im}{\operatorname{Im}}

\numberwithin{equation}{section}

\title{Solving Equation Systems in $\omega$-categorical Algebras} 
\author{Manuel Bodirsky and Thomas Quinn-Gregson}


\begin{document}

\thanks{
Both authors have received funding from  the Deutsche  Forschungsgemeinschaft (DFG)
and from the European Research Council (Grant Agreement no. 681988, CSP-Infinity).
}

\begin{abstract}
We study the computational complexity
of deciding whether a given set of term equalities and inequalities has a solution in an $\omega$-categorical algebra $\mathfrak{A}$. 
There are $\omega$-categorical groups 
where this problem is undecidable. 
We show that if $\mathfrak{A}$ is an $\omega$-categorical  
semilattice or an abelian group, then 
the problem is in P or NP-hard. 
The hard cases are precisely those where
$\Pol(\mathfrak{A},\neq)$ has a uniformly continuous minor-preserving map to the clone of projections on a two-element set. 
The results provide information about
algebras $\mathfrak{A}$ such that $\Pol(\mathfrak{A},\neq)$ 
does not satisfy this condition, 
and they are of independent interest in universal algebra. In our proofs we rely on the Barto-Pinsker theorem about the existence of pseudo-Siggers polymorphisms. To the best of our knowledge, this is the first time that the pseudo-Siggers identity has been used to prove a complexity dichotomy. 
\end{abstract}

\maketitle

\section{Introduction}
The problem of deciding whether a given 
system of linear equations has a solution in ${\mathbb Z}_p$ is one of the central computational problems that can be solved in polynomial time, for example by Gaussian elimination. The problem can also be rephrased as follows: fix the structure $({\mathbb Z}_p;+,0,1,\dots,p-1)$ where $+$ is the binary addition operation and $0,1,\dots,p-1$ are constants;
the problem is then to decide whether a given conjunction of atomic formulas in the signature of this structure is satisfiable in this structure. 
Analogous computational problems can be formulated for other algebraic structures $\bA$ instead of $({\mathbb Z}_p;+,0,\dots,p-1)$, 
and have been studied systematically in the special cases of groups~\cite{GoldmannRussell}, monoids~\cite{MooreTessonTherien}, 
and semigroups~\cite{KlimaTessonTherien}. 

An even more general class of computational problems is the class of \emph{constraint satisfaction problems (CSPs)}; here we fix a structure $\bA$ with a finite signature $\tau$, 
and the task is to decide whether a given 
conjunction of atomic $\tau$-formulas is satisfiable in $\bA$. 
This problem, denoted by $\Csp(\bA)$, is 
typically introduced only for relational signatures; the restriction to relational signatures is not severe, because we may replace each operation $f$ of arity $k$ in $\bA$ by the $k+1$-ary relation $R_f := \{(x_1,\dots,x_k,x_0) \mid x_0 = f(x_1,\dots,x_k)\}$. Then every atomic formula over $\bA$ can be translated into a finite set of atomic formulas in the new signature to obtain a satisfiability-equivalent instance in the new signature. We might have to introduce some additional variables to eliminate nested terms in atomic formulas, but the overall reduction changes the size of the input only by a linear factor.  

It has been conjectured by
Feder and Vardi~\cite{FederVardi} that 
 CSPs 
for fixed structures $\bA$ with a finite domain have a \emph{complexity dichotomy} in the sense that they are either in P or NP-complete. The dichotomy conjecture has been confirmed recently, independently by Bulatov~\cite{BulatovFVConjecture} and by Zhuk~\cite{ZhukFVConjecture}. 
This achievement has been made possible because of an important link between constraint satisfaction and central topics in universal algebra; see, e.g., the survey articles in~\cite{Dagstuhl2017}. 

There are many famous computational problems
that can be phrased as solving equation systems over algebraic structures $\bA$ with an infinite domain; for example, $\Csp(\bA)$ 
for the structure $\bA = ({\mathbb Z};+,\cdot,1)$ 
 is Hilbert's tenth problem, and known to be undecidable~\cite{Matiyasevich}, whereas
 the problem can be solved in polynomial time
 for $\bA = ({\mathbb Z};+,1)$ (see, e.g.,~\cite{Schrijver}). 
In full generality, the mentioned connection between constraint satisfaction and universal algebra breaks down (see the survey article~\cite{Numeric-Domains}). 
However, if the structure $\bA$ is $\omega$-categorical, i.e., if all countably infinite models of the first-order theory of $\bA$ are isomorphic,
then the universal-algebraic approach is still applicable~\cite{BodirskyNesetrilJLC,BartoPinskerDichotomy}. Note that when studying the CSP of infinite-domain structures $\bA$ we still require the signature of $\bA$ to be finite. In particular, we no longer have constants for every element in the domain. If the signature contains no constants at all, then solving equation systems becomes trivial for many algebraic structures: 
for instance for monoids with unit element $1$, we might satisfy all the equations by setting all variables to $1$. 
The natural signature for studying the problem of
solving equations over infinite domains is to additionally allow inequalities in the input, i.e., atomic formulas of the form $s \neq t$ where $s$ and $t$ are terms. In this article we study
problems of the form $\Csp(\bA,\neq)$ where $\bA$ is a finite-signature algebra\footnote{An \emph{algebra} is simply a structure with a purely functional signature; see~\cite{HodgesLong} for basic terminology.}. 
For example, for a given monoid
$\bA$, the problem $\Csp(\bA,\neq)$ is non-trivial in general since we may no longer map all the variables in the input to $1$.

\subsection{Applications}
If $\Csp(\bA,\neq)$ can be solved in polynomial time, then various other interesting computational problems can be solved in polynomial time, too. 
%If $\fA$ is an algebra such that
%$\Csp(\fA,\neq)$ is in P, 
%then this implies polynomial-time algorithms for various other computational problems that have been studied for algebras $\fA$. 
Let $\bA$ be an algebra with a finite signature $\tau$ and a (finite or infinite) domain $A$.  
The \emph{Identity Checking Problem (for $\bA$)}
is the problem of deciding whether for given $\tau$-terms $s,t$ over the variables $x_1,\dots,x_n$ the identity 
$s \approx t$ is 
\emph{valid} in $\bA$, i.e., whether 
\begin{align} 
\bA \models \forall x_1,\dots,x_n \colon  s(x_1,\dots,x_n) = t(x_1,\dots,x_n).
\label{eq:validity}
\end{align}
Note that this is the case if and only if 
there are no elements $a_1,\dots,a_n \in A$ such that
$\bA \models s(a_1,\dots,a_n) \neq t(a_1,\dots,a_n)$; by introducing additional variables and equations we can translate this
into an instance of $\Csp(\bA,\neq)$ which is unsatisfiable if and only if (\ref{eq:validity}) holds. 
Hence, if $\Csp(\bA,\neq)$ is in NP, then
the Identity Checking Problem for $\bA$ is in coNP,
and if $\Csp(\bA,\neq)$ is in P, then
the Identity Checking Problem for $\bA$ is in P, too. 

In the so-called \emph{Entailment Problem (for $\bA$)} we are  
given a finite set of equations $s_1 = t_1,\dots, s_m = t_m$ 
and another equation $s_0 = t_0$ over a common set of variables $V$, and the question 
is whether every assignment $V \to A$ that satisfies  
 $s_1=t_1 \wedge \cdots \wedge s_m=t_m$ also satisfies $s_0 = t_0$. 
Note that this is the case if and only if
the formula $s_1=t_1 \wedge \cdots \wedge s_m=t_m \wedge s_0 \neq t_0$ is unsatisfiable,
so again the problem reduces in polynomial time to the complement of 
$\Csp(\bA,\neq)$. 

Finally, there is a strong link between 
$\Csp(\bA,\neq)$ and the problem $\Csp(\bA,a_1,\dots,a_n)$, where $a_1,\dots,a_n \in A$ are constants, if the algebra $\bA$ is \emph{model-complete}. The notion of model-completeness  is a central concept from model theory and can be seen as a weak form of quantifier elimination: $\bA$ is model-complete if 
every first-order sentence is equivalent to an existential sentence over $\bA$. 
It follows from results in~\cite{Cores-journal,BodHilsMartin-Journal} that 
if $\bA$ is model-complete, then 
for all $a_1,\dots,a_n \in A$ the problem 
$\Csp(\bA,\neq)$ and the problem $\Csp(\bA,\neq,a_1,\dots,a_n)$ are polynomial-time equivalent; in particular, there is a polynomial-time reduction from $\Csp(\bA,a_1,\dots,a_n)$ to 
$\Csp(\bA,\neq)$. Conversely, we will see that 
if $\bA$ satisfies an additional assumption, namely\footnote{This property is equivalent to a property that is often referred to as \emph{convexity} in the theoretical computer science literature~\cite{NelsonOppen80}  and will play an important role in this article.} that there is an embedding $\bA^2 \hookrightarrow \bA$ (i.e., an isomorphism between $\bA^2 = \bA \times \bA$ and a substructure of $\bA$), then there are $a_1,\dots,a_n \in A$ such that there is a polynomial-time 
reduction from $\Csp(\bA,\neq)$ to $\Csp(\bA,a_1,\dots,a_n)$ (Proposition~\ref{prop:constants}). 

\subsection{Results}
We initiate the study the computational complexity of $\Csp(\bA,\neq)$ for $\omega$-categorical algebras $\bA$. 
We first observe that there are $\omega$-categorical groups $\bA$ such that $\Csp(\bA,\neq)$ is undecidable (Section~\ref{sect:undec}). For abelian $\omega$-categorical groups, however, we show that $\Csp(\bA,\neq)$ is in P or NP-complete (Theorem~\ref{thm:abelian-groups}). Recall that if P and NP are distinct then there are also problems in NP that are of intermediate complexity, i.e., neither in P nor NP-hard~(\cite{Ladner}). We also show a P versus NP-hard complexity dichotomy for $\omega$-categorical semilattices (Theorem~\ref{thm:semilattice}).

\subsection{Outline}
In our proofs we rely on recent 
universal-algebraic results for $\omega$-categorical structures, in particular from~\cite{wonderland,BKOPP-equations,BartoPinskerDichotomy}, so we start
by giving a self-contained introduction to the universal-algebraic approach in Section~\ref{sect:ua}. 
Universal-algebraic concepts are also needed 
to precisely state the border between the NP-hard and the polynomial cases in our results.
%, which will be stated at the end of Section~\ref{sect:ua}. 
%which we believe is a dichotomy of independent interest in universal algebra, so we already discuss the link with universal algebra in the following section of the introduction. 
In Section~\ref{sect:algebras} we specialise the
universal-algebraic approach to structures of the
form $(\bA,\neq)$ where $\bA$ is an algebra,
and in Section~\ref{sect:monoids} we specialise further to monoids.  
%Our results for semi-lattices and our result for Abelian groups use some general facts for monoids, which we present in Section~\ref{sect:monoids}. 
% and Section~\ref{sect:lattices} contains our
%classification for $\omega$-categorical lattices.
%This section also contains a polynomial-time algorithm for $\Csp(\bA;\neq)$ where $\bA$ is the countable universal homogeneous lattice, which seems to be new and might have applications in computer science. 
Section~\ref{sect:groups} contains our
classification for $\omega$-categorical abelian groups. 
%The polynomial-time tractable cases in this section use the Bulatov-Dalmau algorithm for Maltsev constraints over finite domains~\cite{Maltsev}. 
Finally, Section~\ref{sect:semilattices} contains our classification for $\omega$-categorical 
semilattices.
We close with a discussion and some open problems in Section~\ref{sect:discussion}. 

%\subsection{Universal-Algebraic Results}

\section{The Universal-Algebraic Approach}
\label{sect:ua}

%\subsection{Polymorphisms} 
The universal-algebraic approach is based on the following concept from universal algebra.
An operation $f \colon A^k \to A$ \emph{preserves}
a relation $R \subseteq A^m$ if for all $t_1,\dots,t_k \in R$ the $m$-tuple 
$f(t_1,\dots,t_k)$ obtained from applying $f$ componentwise is also contained in $R$. 
Note that if $g \colon A^m \to A$ is an operation, then $f$ preserves the \emph{graph $R_g$ of $g$}, defined as 
$$R_g := \{(a_1,\dots,a_{m},g(a_1,\dots,a_m)) \mid a_1,\dots,a_{m} \in A\},$$
if and only if 
%Note that $f$ is a polymorphism of an algebra $\fA$ if and only if 
\emph{$f$ commutes with $g$}, i.e., for all $a_{1,1},\dots,a_{n,m} \in A$ 
\begin{align*}
& f(g(a_{1,1},\dots,a_{1,m}),\dots,g(a_{n,1},\dots,a_{n,m})) \\
= \; & g(f(a_{1,1},\dots,a_{n,1}),\dots,f(a_{1,m},\dots,a_{n,m})).
\end{align*}
In this case we say that \emph{$f$ preserves $g$}. 
An operation $f$ is a \emph{polymorphism} of a structure $\bA$ if $f$ preserves all relations and all operations of $\bA$. 
Note that the projections $\pi^k_i \colon A^k \to A$ 
defined by $\pi^k_i(a_1,\dots,a_k) := a_i$ 
is a polymorphism of every structure with domain $A$. We also would like to mention that similarly as the set of all automorphisms of $\bA$ forms a group, the set of all polymorphisms of $\bA$, denoted by $\Pol(\bA)$, forms a \emph{clone}, i.e., the set of polymorphisms is closed
under composition and contain the projections. The clone of projections on a two-element set will be denoted by $\Proj$. A map between two clones
is called \emph{minor-preserving} if 
it maps operations to operations of the same arity, and if $$\xi(f(p_1,\dots,p_n)) = \xi(f)(p_1,\dots,p_n)$$ for all $n$-ary operations $f$ and projections $p_1,\dots,p_n$ of the same arity $m$.  
In the introduction we have mentioned that 
$\Csp(\bA)$ is for every finite structure $\bA$ with finite relational signature in P or NP-complete; using polymorphisms, the border between the two cases can be stated as follows, combining results from~\cite{Siggers,JBK,ZhukFVConjecture,BulatovFVConjecture,wonderland}. 

\begin{theorem}
Let $\bA$ be a structure with finite domain and finite relational signature. 
Then either 
\begin{itemize}
\item $\bA$ has a polymorphism $s \colon A^6 \to A$
which is \emph{Siggers}, i.e., satisfies
$$s(x,y,x,z,y,z) \approx s(y,x,z,x,z,y)$$
in this case,  $\Csp(\bA)$ is in P, or
\item $\Pol(\bA)$ has a minor-preserving map to $\Proj$; in this case, $\Csp(\bA)$ is NP-complete. 
\end{itemize}
\end{theorem}

%Then discuss the infinite;
%motivation for $\omega$-categorical case. 

The fact that $\Csp(\bA)$ is NP-hard if $\bA$ does not have a Siggers polymorphism~\cite{Siggers}  
was already known before the break-through result from~\cite{ZhukFVConjecture,BulatovFVConjecture}. The equivalence of the existence of a Siggers polymorphism and of the non-existence of a minor-preserving map to $\Proj$ is from~\cite{wonderland}. 
For general $\omega$-categorical structures, the equivalence is no longer valid~\cite{BKOPP-equations}, but we still have the  following hardness condition. 

\begin{theorem}[\cite{wonderland}]\label{thm:uch1}
Let $\bA$ be an $\omega$-categorical structure with a finite relational$^{\ref{note1}}$ signature. If $\Pol(\bA)$ has a 
uniformly continuous\footnote{In our setting, $\xi \colon {\mathscr C} \to \Proj$ is uniformly continuous 
if and only if there exists a finite set $F \subseteq C$ such that if $f,g \in \mathscr C$
agree on $F$, then $\xi(f) = \xi(g)$.} 
minor-preserving map to $\Proj$ 
then $\Csp(\bA)$ is NP-hard. 
\end{theorem}

To apply this hardness condition,  
we need the following terminology from~\cite{Cores-journal}. 
An $\omega$-categorical structure $\bA$ is called a \emph{core} if every endomorphism of $\bA$ (i.e., every homomorphism from $\bA$ to $\bA$)
is an embedding. 
%Model-completeness has already been introduced in the introduction.
Two structures $\bA$ and $\bB$ are called \emph{homomorphically equivalent} if there is a
homomorphism from $\bA$ to $\bB$ and vice versa. Clearly, two structures that are homomorphically equivalent have the same CSP. 

\begin{theorem}[\cite{Cores-journal,BodHilsMartin-Journal}]\label{thm:mc-core}
Every $\omega$-categorical relational structure $\bB$ is homomorphically equivalent to a model-complete core structure $\bC$, which is unique up to isomorphism, and again $\omega$-categorical,
and which will be called the \emph{model-complete core} of $\bC$. 
\end{theorem}

A first-order formula $\phi$ is called \textit{primitive positive} if it is of the form 
\[ \exists \bar{x} ( \phi_1(\bar{x}) \wedge \cdots \wedge \phi_n(\bar{x})), 
\] 
where $\phi_1,\dots,\phi_n$ are atomic formulas. Every primitive positive relation of a relational structure $\bB$ is preserved by the polymorphisms of $\bB$. If $\bB$ is $\omega$-categorical then conversely every relation left invariant by polymorphisms of $\bB$ is primitive positive definable \cite{BodirskyNesetrilJLC}. 

The following is implied by results in~\cite{wonderland}.

\begin{proposition}\label{prop:coreh1}
Let $\bB$ be an $\omega$-categorical relational\footnote{\label{note1}It will be explained in Remark~\ref{rem:gen-sig} that the result also holds for general structures that also might contain operations.} structure. 
\begin{itemize}
\item If $\bC$ is homomorphically equivalent to $\bB$ 
then there is a uniformly continuous minor-preserving map from $\Pol(\bB)$ to $\Pol(\bC)$. 
\item If $\bC$ is the model-complete core
of $\bB$ and $c_1,\dots,c_n \in C$, then $\Pol(\bC)$ (and $\Pol(\bB)$) has a uniformly continuous minor-preserving map to $\Pol(\bC,c_1,\dots,c_n)$.  
\item If $\bA$ is a substructure of $\bB$ whose domain is primitive positive definable in $\bB$, then
 there is a uniformly continuous minor-preserving map from $\Pol(\bB)$ to
 $\Pol(\bA)$. 
\end{itemize}
\end{proposition}
 
Hence, if $\Pol(\bC,c_1,\dots,c_n)$
or $\Pol(\bA)$ in Proposition~\ref{prop:coreh1}
 has a uniformly continuous minor-preserving map to 
$\Proj$, then $\Csp(\bB)$ is NP-hard by Theorem~\ref{thm:uch1}, because
the composition of uniformly continuous minor-preserving maps is uniformly continuous and minor-preserving. For model-complete cores, 
we will use the following result.

% of Barto and Pinsker~\cite{BartoPinskerDichotomy}. 
%the condition from Theorem~\ref{thm:uch1} can be rephrased as follows, which will be most important in this article. 

\begin{theorem}[Barto and Pinsker~\cite{BartoPinskerDichotomy}]
\label{thm:BP}
Let $\bC$ be an $\omega$-categorical relational$^{\ref{note1}}$ structure 
which is a model-complete core. 
Then at least one of the following holds. 
\begin{itemize}
\item $\bC$ has a pseudo-Siggers polymorphism,
i.e., a polymorphism $s \colon C^6 \to C$ 
and endomorphisms $e_1,e_2 \colon C \to C$
satisfying $$e_1\big(s(x,y,x,z,y,z) \big) \approx e_2 \big (s(y,x,z,x,z,y) \big).$$
\item $\Pol(\bC)$ has a uniformly continuous
minor-preserving map to $\Proj$.  
%there are finitely many elements $b_1,\dots,b_n$ such that 
%$\Pol(\bC,b_1,\dots,b_n)$ has a continuous homomorphism to $\Proj$. 
\end{itemize}
\end{theorem}
%Note that
%the first item in Theorem~\ref{thm:BP} is topology-free; it will be most important in this article. 
In this article we will show how to use the pseudo-Siggers identity to obtain structural information about $\bC$ 
if $\bC$ is of the form $(\bA,\neq)$ where $\bA$ is an $\omega$-categorical semilattice or 
%lattice, or   
abelian group.

\begin{remark}
In many situations, the two
items in Theorem~\ref{thm:BP} are mutually exclusive; two general conditions that imply this have been presented in~\cite{BKOPP-equations}. However, these conditions do not cover our setting, not even in the special case of 
semilattices. Abelian groups are covered, but this requires an extra argument that will be given in Section~\ref{sect:abelian}. 
\end{remark}  

\section{Algebras}
\label{sect:algebras}
An algebra $\bA$ is a structure with domain $A$ and with a purely functional signature. The $n$-ary polymorphisms of $\bA$ are precisely the (algebra) homomorphisms $g:A^n\rightarrow A$. 
In this section we make some observations
that are relevant for the universal-algebraic approach to the CSP of structures of the form $(\bA,\neq)$. 

{\bf Conventions.} 
We write $\omega = \{0,1,2,\dots\}$ for the set of natural numbers including zero.
The equality symbol is always allowed in first-order formulas.

\subsection{Homogeneity}
An important source of $\omega$-categorical
structures comes from Fra\"{i}ss\'e-amalgamation.
The \emph{age} of a $\tau$-structure is
the class of all finitely generated $\tau$-structures
that embed into the structure. A structure
is called \emph{homogeneous} if every
isomorphism between finitely generated substructures extends to an automorphism. 
Let $\tau$ be a countable signature and 
let $\mathcal K$ be a class of finitely generated 
$\tau$-structures which is closed under subalgebras, has the joint embedding property and the amalgamation property, and contains
countably many isomorphism types of structures. 
Then there exists a countable homogeneous
$\tau$-structure
$\mathscr F$  whose age is ${\mathcal K}$ (Theorem 6.1.2.~in~\cite{Hodges}). 
A structure $\bA$ is called \emph{uniformly locally finite} if there exists 
a function $f \colon \omega \to \omega$
such that every substructure of $\bA$ generated
by $n$ elements has at most $f(n)$ elements. 
If $\bA$ is $\omega$-categorical
then $\bA$ must be uniformly locally finite.
Conversely, every homogeneous uniformly locally finite structure is $\omega$-categorical (\cite{Hodges}, Corollary 6.2). 

%A class is called uniformly locally finite
%for every substructure $\bB$ of $\bA$
%The class ${\mathcal K}$ is said to be 
%\emph{uniformly locally finite} if there exists 
%a function $f \colon \omega \to \omega$
%such that for every 

\subsection{Model companions}
Let $\bA$ and $\bB$ be algebras with the same signature. 
Note that homomorphisms between
$(\bA,\neq)$ and $(\bB,\neq)$ must be embeddings (which is not true in general if $\bA$ and $\bB$ are arbitrary structures). It follows that structures of the form
$(\bA,\neq)$ must be cores. 
Two structures $\bA$ and $\bB$ are called \emph{companions}
if they satisfy the same universal first-order sentences (for instance, if $\bA$ is a semilattice, so is every companion of $\bA$). 
Note that in this case, $\bA$ and $\bB$ have the same age. 
The implication from (1) to (2) in the following
lemma can be shown
by a compactness argument (see, e.g., \cite{Cores-journal}); it is straightforward to prove the other implications in cyclic order. 

\begin{lemma}\label{lem:c}
Let $\bA$ and $\bB$ be $\omega$-categorical algebras. Then the following are equivalent. 
\begin{enumerate}
\item $\bA$ and $\bB$ are companions;
\item $\bA \hookrightarrow \bB$ and $\bB \hookrightarrow \bA$; 
\item $(\bA,\neq)$ and $(\bB,\neq)$ are homomorphically equivalent;
\item $\Csp(\bA,\neq)$ and $\Csp(\bB,\neq)$ are
the same computational problem;
\item $\Age(\bA) = \Age(\bB)$. 
%\item a finite structures embeds into $\bA$ if
%and only if it embeds into $\bB$.
\end{enumerate}
\end{lemma}

A structure $\bB$ is called a \emph{model companion of $\bA$} if $\bA$ and $\bB$ are companions and $\bB$ is 
model-complete. 
Every $\omega$-categorical structure has a model companion~\cite{Saracino}, which is
unique up to isomorphism and $\omega$-categorical (see, e.g.,~\cite{HodgesLong}). 
For illustration, we present an example of an $\omega$-categorical algebra and its model companion. 

\begin{example}
For $a,b \in {\mathbb Q}$ 
we write $[a,b]$ for $\{x \in {\mathbb Q} \mid a \leq x \leq b\}$ and $\min$ for the binary operation
that returns the minimum of its two arguments. 
Then $([0,1];\min)$ and $({\mathbb Q};\min)$ are companions. Since $({\mathbb Q};\min)$ is model-complete, it is the model companion of $([0,1];\min)$.
\end{example} 

%The algebra $([0,1];\min)$ is not model-complete, because 
%\begin{itemize}
%\item 
%\item $({\mathbb Q};\min)$ is model-complete;
%\item $([0,1];\min)$ and $({\mathbb Q};\min)$ are not isomorphic. 
%\end{itemize}
%\end{example}

%Note that in Theorem~\ref{thm:mc-core}
%we assumed that the signature is relational;

Unfortunately, several of the results
that we cited in Section~\ref{sect:ua}
were originally only formulated for relational signatures. But it is not difficult to see 
that they also hold for structures that might involve operations, as we will see in the following. The definition of model-complete cores for general $\omega$-categorical structures $\bB$ is the same as the one we gave for the relational case:
a structure $\bC$ is \emph{a model-complete core of $\bB$}  
if $\bC$ and $\bB$ are homomorphically equivalent and $\bC$ is a model-complete core. 

Let $\bB$ be a structure. 
We write $\bB^*$ for 
the relational structure obtained from $\bB$ by
replacing each operation $g$ in $\bB$ of arity $k$ 
by a relation symbol $R_g$ of arity $k+1$ that denotes in $\bB^*$ the graph 
of the operation $g^{\bB}$. 

\begin{remark}
There are homogeneous algebras $\bA$ such that $\bA^*$ is not homogeneous: for example, consider the group $\bA := {\mathbb Z}_2 \times {\mathbb Z}_3$ generated by an element $a$ of order 2 and an element $b$ of order 3. Then it is easy to verify that $\bA$ is homogeneous in the signature $\{\cdot\}$ of semigroups, but in $\bA^*$ 
the substructures induced by $\{a\}$ and $\{b\}$ 
are isomorphic, and no automorphism of $\bA^*$ maps $a$ to $b$. 
\end{remark}

\begin{lemma}
\label{lem:mc}
Let $\bB$ be an $\omega$-categorical structure.
Then $\bB$ has a model-complete core $\bC$, 
%is homomorphically equivalent to 
%a structure $\bC$ which is a model-complete core,
which is unique up to isomorphism and again $\omega$-categorical. Moreover,
$\bC^*$ is the model-complete core of $\bB^*$.  
\end{lemma}
\begin{proof}
By Theorem~\ref{thm:mc-core}, the relational structure $\bB^*$ has a model-complete core $\bC'$ which is $\omega$-categorical. 
Since $\bC'$ and $\bB^*$ are homomorphically equivalent, there are 
homomorphisms $h \colon \bB^* \to \bC'$
and $i \colon \bC' \to \bB^*$. 
For each $k$-ary function symbol $g$ from the signature $\tau$ of $\bB$, the
relation denoted by $R_g$ in $\bC'$ is the graph of a $k$-ary operation on $C'$. Indeed\footnote{This would not be true for arbitrary structures $\bC'$ that are homomorphically equivalent to $\bB^*$, but we will use model-completeness.} let 
$\phi(x_1,\dots,x_n)$ be the formula
$\exists z \colon R_g(x_1,\dots,x_n,z)$ and let $u_1,\dots,u_n \in C'$. Then there exists $z \in B$ such that $g^{\bB}(i(u_1),\dots,i(u_n)) = z$. 
Hence, $(i(u_1),\dots,i(u_n),z) \in R^{\bB^*}_{g}$
and thus $(h \circ i(u_1),\dots,h \circ i(u_n),h(z)) \in 
R^{\bC'}_g$, so $\bC' \models \phi(h \circ i(u_1),\dots,h \circ i(u_n))$ and thus $\bC' \models
\phi(u_1,\dots,u_n)$. Moreover, if
$(u_1,\dots,u_n,a),(u_1,\dots,u_n,b) \in R_g^{\bC'}$ then $(i(u_1),\dots,i(u_n),i(a)), (i(u_1),\dots,i(u_n),i(b)) \in R_g^{\bB^*}$, so 
$i(a) = i(b)$ since $R_g^{\bB^*}$ is the graph of the function $g^{\bB}$. Thus, $a=b$ because $i$ is injective. 

Let $\bC$ be the $\tau$-structure 
with the same domain and relations as $\bC'$ and 
such that every operation symbol 
$g \in \tau$ 
denotes the operation whose graph is
$R^{\bC'}_g$. Clearly, $\bC^*$ equals $\bC'$. We prove that $\bC$ is a model-complete core of $\bB$: 
the maps $h$ and $i$ are homomorphisms
from $\bB$ to $\bC$ and from $\bC$ to $\bB$, respectively, showing that $\bB$ and $\bC$ are homomorphically equivalent.  
Every endomorphism of $\bC$ is an endomorphism of $\bC'$, and hence 
preserves all first-order formulas over $\bC'$ and also preserves all first-order formulas over $\bC$. So $\bC$ is  
a model-complete core. 

If $\bD$ is a model-complete core that is homomorphically equivalent with $\bB$, then $\bD^*$ is homomorphically equivalent to $\bC^*$, and hence $\bD^*$ and $\bC^*$ are isomorphic. It follows that
$\bD$ and $\bC$ are isomorphic, showing
the uniqueness of $\bC$ up to isomorphism. 
\end{proof}

As in the relational case, because
of the uniqueness of the model-complete
core up to isomorphism we call 
$\bC$ \emph{the} model-complete core of $\bB$. 

\begin{remark}\label{rem:gen-sig}
Similarly as in the proof of Lemma~\ref{lem:mc} it can be shown that the assumption in Theorem~\ref{thm:uch1}, Proposition~\ref{prop:coreh1}, and Theorem~\ref{thm:BP}
that the structures are relational can be dropped. 
\end{remark}

\begin{corollary}
Let $\bA$ be an $\omega$-categorial algebra and $\bC$ its model companion. 
Then $(\bC,\neq)$ is the model-complete core of $(\bA,\neq)$. 
\end{corollary}
\begin{proof}
The structure $(\bC,\neq)$ 
is a model-complete core and homomorphically equivalent to $(\bA,\neq)$ by Lemma~\ref{lem:c}. 
So the model-complete core of $(\bA,\neq)$
must be isomorphic to $(\bC,\neq)$. 
%Let $N$ be the binary relation
%of $\bB$ that is denoted by the symbol $\neq$ for inequality; we have to show that $N = \{(x,y) \in B^2 \mid x \neq y\}$. If $a,b \in B$ are such that $(a,b) \in N$ then $i(a) \neq i(b)$, hence $a \neq b$. 
%Now suppose that $a,b \in B$ are distinct. 
%Since $h \circ i$ is an endomorphism of $\bB$ and $\bB$ is a model-compete core, $h \circ i$ preserves all first-order formulas, and in particular the formula $\neg (x=y)$. It follows that 
%$h \circ i(a) \neq h \circ i(b)$ and thus $i(a) \neq i(b)$. This in turn implies that $(h \circ i(a),h \circ i(b)) \in N$. Again using the assumption that $h \circ i$ preserves all first-order formulas, we
%obtain that $(a,b) \in N$ as desired. 
\end{proof}

\subsection{Square embeddings}
\label{sect:square-emb}
In theoretical computer science~\cite{NelsonOppen80}, a first-order $\tau$-theory 
$T$ is called \emph{convex} if for every finite set of atomic $\tau$-formulas $S$ the set 
$$T \cup S \cup \{x_1 \neq y_1,\dots,x_m \neq y_m\}$$ 
is satisfiable if and only if $T \cup S \cup \{x_i \neq y_i\}$ is satisfiable for each $i \in \{1,\dots,m\}$.
If $\bA$ is a structure, we write $\Th(\bA)$ for 
the first-order theory of $\bA$, i.e., for the set of all first-order sentences that hold in $\bA$. 
If $T = \Th(\bA)$ for some structure $\bA$, then an alternative terminology~\cite{BroxvallJonssonRenz} for convexity is that $\neq$ is \emph{1-independent from $\bA$}.  

\begin{proposition}
\label{prop:convex}
Let $\bA$ be an $\omega$-categorical algebra. 
Then the following are equivalent. 
\begin{enumerate}
\item $\Th(\bA)$ is convex;
\item $\bA$ has a binary injective polymorphism;
\item $\bA^2 \hookrightarrow \bA$; 
\item $\bA^k \hookrightarrow \bA$ for all $k \in \omega$; 
\item $\Age(\bA)$ is closed under finite direct products.
\end{enumerate}
\end{proposition}
\begin{proof}
The equivalence of $(1)$ and $(2)$ is shown for relational $\omega$-categorical structures~\cite{Bodirsky-HDR}, and the same proof also works for $\omega$-categorical structures with functions. 
The implication from $(2)$ to $(3)$ holds
because $\bA$ is an algebra. 
The implications from (3) to (4) and
from (4) to (5) are clear. 
For the implication from (5) to (1), 
suppose that 
$\Th(\bA) \cup S \cup \{x_i \neq y_i\}$ is satisfiable for each $i \in \{1,\dots,m\}$. 
Let $\bA_i$ be the substructure of $\bA$ induced by the variables of $S$ and 
$\{x_i,y_i\}$. Then by assumption 
$\bA_1 \times \cdots \times \bA_m$ is a substructure of $\bA$ and witnesses that
$\Th(\bA) \cup S \cup \{x_1 \neq y_1,\dots,x_m \neq y_m\}$ 
is satisfiable. 
\end{proof}

We present a pair of applications of
square embeddings in the context of equation solving. 

%Question: does this automatically give totally symmetric polymorphisms of all arities?
%Converse is false: take, e.g., $({\mathbb Q};\min)$ which is preserved by $\min$. However,
%do we have a converse if $\bA$ is a model-complete core algebra?  

\begin{proposition}\label{prop:constants}
Let $\bA$ be a model-complete $\omega$-categorical structure with finite signature $\tau$ such that $\bA^2 \hookrightarrow \bA$. 
Then there are finitely many $a_1,\dots,a_n \in A$ such that 
$\Csp(\bA,\neq)$ is polynomial-time equivalent
to $\Csp(\bA,a_1,\dots,a_n)$. 
\end{proposition}
\begin{proof}
We have already mentioned in the introduction
that if $\bB$ is a model-complete $\omega$-categorical structure then for all $a_1,\dots,a_n \in B$ there is a polynomial-time reduction from $\Csp(\bB,a_1,\dots,a_n)$ to $\Csp(\bB)$; see~\cite{Cores-journal,BodHilsMartin}. 
For the converse reduction, note that 
for every conjunction of atomic $\tau$-formulas $\phi$ the formula 
$$\phi \wedge s_1 \neq t_1 \wedge \cdots \wedge s_m \neq t_m$$
is satisfiable in $\bA$ if and only if $\phi \wedge s_i \neq t_i$ is satisfiable in $\bA$ for each $i \in \{1,\dots,m\}$, because $\Th(\bA)$ is convex. 
By introducing new variables and new identities in $\phi$, we may assume that each of the conjuncts $s_i \neq t_i$ is in fact of the form $x_i \neq y_i$ for variables $x_i$ and $y_i$.
%Let $k$ be the maximal arity of all the function symbols in $\tau$. 
To test whether $\phi \wedge x_i \neq y_i$ is satisfiable, we pick representatives $a_1,\dots,a_l$ for each orbit of pairs in $\Aut(\bA)$. 
Note that $\phi \wedge x_i \neq y_i$ is satisfiable
if and only if $\phi \wedge x_i = a \wedge y_i = a'$ is satisfiable in $\bA$ for some orbit representatives $a,a' \in \{a_1,\dots,a_l\}$. Hence,
$\Csp(\bA,\neq)$ can be reduced 
to $\Csp(\bA,a_1,\dots,a_l)$. The reduction is in  AC$_0$, and in particular in Logspace and Ptime. 
\end{proof} 

\begin{proposition}\label{prop:newproduct} Let $\bA$ and $\bB$ be $\omega$-categorical algebras with the same signature $\tau$ such that $\bA^2 \hookrightarrow \bA$.
 If $\Csp(\bA,\neq)$ is in P  then there is a polynomial-time reduction from $\Csp(\bA\times \bB,\neq)$  to $\Csp(\bB,\neq)$. 
\end{proposition} 

\begin{proof} Let $\phi$ be a conjunction of atomic $\tau$-formula and consider formula
\[ \Phi:=\phi \wedge x_1\neq y_1 \wedge \cdots \wedge x_m\neq y_m
\] 
with variables $V$, where we again assume without loss of generality that $x_i,y_i\in V$. 
We claim that $\Phi$ is satisfiable in $\bA \times \bB$ if and only if there exists a partition $\{1,2,\dots,m\}=I\cup J$ such that $\Phi_I= \phi \wedge \bigwedge_{i\in I} x_i\neq y_i$  is satisfiable in $\bA$ and $\Phi_J = \phi \wedge \bigwedge_{j\in J} x_i\neq y_i$   is satisfiable in $\bB$. 

Let $f\colon V\rightarrow A\times B$ be an assignment satisfying $\Phi$. Let $f_A\colon V\rightarrow A$ be given by $f_A(v)=a$ if and only if $f(v)=(a,b)$ for some $b\in B$; dually define $f_B$. Hence $f(v)=(f_A(v),f_B(v))$. 
 Then  $f_A$ and $f_B$ satisfies $\phi$, and if $f(x_i) \neq f(y_i)$ then either $f_A(x_i)\neq f_A(y_i)$ or $f_B(x_i)\neq f_B(y_i)$ (or both).   
  Letting $I=\{k \colon f_A(x_i)\neq f_A(y_i)\}$ and $J=\{k\colon f_A(x_i) =  f_A(y_i) \text{ and } f_B(x_i) \neq f_B(y_i)\}$ we obtain desired partition of $\{1,2,\dots,m\}$.  

Conversely, let $V_I$ and $V_J$ be the variables of $\Phi_I$ and $\Phi_J$, respectively. Let $g_A \colon V_I\rightarrow A$ and $g_B\colon V_J\rightarrow B$ be assignments satisfying $\Phi_I$ and $\Phi_J$, respectively. Fix $a\in A$ and $b\in B$. Expand $g_A$ to a map $g'_A \colon V\rightarrow A$ by letting $g'_A(x)=a$ for any $x\in V\setminus V_I$; dually obtain $g'_B$. Then the map $V\rightarrow A\times B$ given by $v\mapsto (g'_A(v),g'_B(v))$ is an assignment satisfying $\Phi$. This finishes the proof of our claim. 

This gives rise to the following  method for determining the satisfiability of $\Phi$. For each $1\leq i \leq m$ we check   (in polynomial-time) if $\Phi_i = \phi \wedge x_i \neq y_i$ is satisfiable in $A$. Let $K$ be the set of $i\in \{1,\dots,m\}$ for which $\Phi_i$ is not satisfied in $A$.  By the claim and Proposition \ref{prop:convex} we have that $\Phi$ is satisfiable in $\bA \times \bB$ if and only if $\phi \wedge \bigwedge_{k\in K} x_k\neq y_k$  is satisfiable in $\bB$. 
\end{proof}

The following lemma shows that the property to have a square embedding implies the existence of a pseudo-Siggers polymorphism in an important situation.

\begin{lemma}\label{lem:square-pseudo-sig}
Let $\bA$ be an $\omega$-categorical algebra.
If 
%\begin{itemize}
%\item $\bA$ is homogeneous and $\bA^2 \hookrightarrow \bA$, or 
%\item 
$\bA^2$ is isomorphic to $\bA$, 
%\end{itemize}
then $\bA$ has a pseudo-Siggers polymorphism. 
\end{lemma}
\begin{proof}
%Suppose first that $\bA^2$ is isomorphic to $\bA$, so 
Clearly there exists an isomorphism $g \colon \bA^6 \to \bA$. Let $\alpha \colon A \to A$ be the map defined as follows.
For $a \in A$, let $(a_1,\dots,a_6) \in A^6$ be such that $g(a_1,\dots,a_6) = a$. Define 
$\alpha(a) := g(a_2,a_1,a_4,a_3,a_6,a_5)$; 
then $\alpha$ is an automorphism of $\bA$, because $g$ is an isomorphism, and for all $x,y,z \in A$ we have 
%\begin{align*} 
that $\alpha(g(x,y,x,z,y,z)) = g(y,x,z,x,z,y)$, 
so $g$ is a pseudo-Siggers polymorphism. 
%\end{align*}
%Then $g$ is a pseudo-Siggers polymorphism of $\bA$, witnessed by the automorphism $\alpha$ of $\bA$ mapping $g(x,y,x,z,y,z)$ to $g(y,x,z,x,z,y)$.  
\end{proof}

\subsection{Pseudo-Siggers polymorphisms}
If $\bA$ is an $\omega$-categorical algebra,
then the existence of a pseudo-Siggers
polymorphism of $(\bA,\neq)$ has an interesting consequence, which is in fact equivalent if the algebra $\bA$ is even homogeneous. 

\begin{lemma}\label{lem:ps}
Let $\bA$ be an $\omega$-categorical algebra. If $s \in \Pol^{(6)}(\bA,\neq)$ is a pseudo-Siggers operation then for all $x,y,z,u,v,w \in A$
\begin{align}
& s(x,y,x,z,y,z) = s(u,v,u,w,v,w) \nonumber \\
\Leftrightarrow \quad & s(y,x,z,x,z,y) = s(v,u,w,u,w,v). \label{eq:siggers}
\end{align}
If $\bA$ is homogeneous, the converse implication holds as well. 
\end{lemma}
\begin{proof}
Let $s$ be a pseudo-Siggers operation, i.e., there are $e_1,e_2 \in \End(\bA,\neq)$ such that  $e_1(s(x,y,x,z,y,z)) = e_2(s(y,x,z,x,z,y))$. Now observe that 
\begin{align*}
& s(x,y,x,z,y,z) = s(u,v,u,w,v,w) \\
\Leftrightarrow \; & e_1(s(x,y,x,z,y,z)) = e_1(s(u,v,u,w,v,w)) && \text{(since $e_1$ is injective)} \\
\Leftrightarrow \; & e_2(s(y,x,z,x,z,y)) = e_2(s(v,u,w,u,w,v)) && \text{(by assumption)} \\
\Leftrightarrow \; & s(y,x,z,x,z,y) = s(v,u,w,u,w,v) && \text{(since $e_2$ is injective)}. 
\end{align*}
%if
%$s(x,y,x,z,y,z) = s(u,v,u,w,v,w)$, then 
%$e_1(s(x,y,x,z,y,z)) = e_1(s(u,v,u,w,v,w))$
%and hence $e_2(s(y,x,z,x,z,y) = e_2(s(v,u,w,u,w,v))$, which in turn implies that $s(y,x,z,x,z,y) = s(v,u,w,u,w,v)$ because $e_2$ is injective. 
%All implications in this argument can be reversed which concludes the first direction of the proof. 

Conversely, suppose that $s$ satisfies~(\ref{eq:siggers}). By the lift lemma (Lemma~3 in~\cite{canonical}) it suffices to show that for
every finite $F \subseteq A$ there exists
$\alpha \in \Aut(\bA)$ such that $s(x,y,x,z,y,z) = \alpha s(y,x,z,x,z,y)$ for all $x,y,z \in F$. 
Since $\bA$ is homogeneous, 
it suffices to verify that for every $k \in {\omega}$ and $a_1,a_2,a_3 \in F^k$ 
the $k$-tuples $s(a_1,a_2,a_1,a_3,a_2,a_3)$
and $s(a_2,a_1,a_3,a_1,a_3,a_2)$ satisfy
the same atomic formulas in the language of $(\bA,\neq)$.
 So let $r,t$ be terms in the language of $\bA$ such that $r(s(a_1,a_2,a_1,a_3,a_2,a_3)) = t(s(a_1,a_2,a_1,a_3,a_2,a_3))$. 
Then \begin{align*}
s(r(a_1),r(a_2),r(a_1),r(a_3),r(a_2),r(a_3)) & = r(s(a_1,a_2,a_1,a_3,a_2,a_3)) \\
& = t(s(a_1,a_2,a_1,a_3,a_2,a_3)) \\
& = s(t(a_1),t(a_2),t(a_1),t(a_3),t(a_2),t(a_3))
\end{align*} and therefore
the assumption implies that 
$$s(r(a_2),r(a_1),r(a_3),r(a_1),r(a_3),r(a_2))
=s(t(a_2),t(a_1),t(a_3),t(a_1),t(a_3),t(a_2)),$$ which in
turn implies that $r(s(a_2,a_1,a_3,a_1,a_3,a_2))=t(s(a_2,a_1,a_3,a_1,a_3,a_2))$. Symmetrically, one can show that every atomic formula that
holds on $s(a_2,a_1,a_3,a_1,a_3,a_2)$ also holds 
on $s(a_1,a_2,a_1,a_3,a_2,a_3)$. 
\end{proof}

We would like to point out that in later sections,
whenever we use the assumption that $\bA$
has a pseudo-Siggers polymorphism, we do this by using Property~(\ref{eq:siggers}), and 
we are not aware of an $\omega$-categorical algebra where the converse of Lemma~\ref{lem:ps} does not hold. 
%\begin{proposition}
%Let $\bA$ be an $\omega$-categorical algebra such that $\bA^2 \hookrightarrow \bA$. Then $(\bA,\neq)$ has a pseudo-Siggers polymorphism. 
%\end{proposition}
%\begin{proof}
%Is it true? 
%\end{proof}

\section{Monoids}
\label{sect:monoids}
Let $\bM = (M;\cdot,1)$ be a monoid.
Polymorphisms $f \colon M^n\rightarrow M$ of $\bM$ have the particularly pleasing property that they decompose in the following sense: for any $x_1,\dots,x_n\in M$ we have 
\begin{align}
f(x_1,x_2,\dots,x_n)=f(x_1,1,\dots,1) \cdot f(1,x_2,1,\dots,1) \cdots f(1,1,\dots,1,x_n). \label{eq:decomp}
\end{align}
For $I \subseteq \{1,\dots,n\}$ we
write $x^{(n)}_I$ for the $n$-tuple 
whose $i$-th component is $x$ if $i \in I$
and $1$ otherwise. 

\begin{definition} 
Let $f \in \Pol^{(n)}(\bM)$. 
Let $f_I \colon M \to M$ be the operation given by $$f_I(x) := f(x^{(n)}_I).$$
\end{definition}

%Note that $f_{\{1,3\}}(x) \cdot f_{\{2,5\}}(y) \cdot f_{\{4,6\}}(z) = s(x,y,x,z,y,z)$. 

\begin{remark}
The unary constant operation $x \mapsto 1$ 
is an endomorphism of every monoid $\bM$. It follows that for every $I \subseteq \{1,\dots,n\}$
and every $f \in \Pol^{(n)}(\bM)$ 
the operation $f_I$ is an endomorphism of $\bM$. Every polymorphism of $(\bM,\neq)$ must preserve $M \setminus \{1\}$. 
\end{remark}

Note that $f_I$ is not necessarily a 
self-embedding of $(M;\cdot,1)$. For example, the projection $f(x,y)=x$ is a polymorphism of $(M;\cdot,1)$, and $f_{\{2\}}$ is constant. However, there must be a subset $I$ of $\{1,\dots,n\}$ such that $f_I$ is an embedding. 

\begin{proposition}\label{prop:monoid-decomp}
Let $(M;\cdot,1,\neq)$ be a monoid and $f \colon M^n \to M$ a polymorphism of $(M;\cdot,1)$. 
Then for any partition $I_1 \cup I_2 \cup \cdots \cup I_k$ of $\{1,\dots,n\}$ there exists $j \in \{1,\dots,k\}$ such that $f_{I_j}$ is a self-embedding of $(M;\cdot,1)$. 
\end{proposition}
\begin{proof}
Let $I_1 \cup I_2 \cup \cdots \cup I_k$ be a partition of $\{1,\dots,n\}$. Suppose for contradiction that 
for every $j \in \{1,\dots,k\}$ there exist
distinct $x_j,y_j \in M$ 
such that $f_{I_j}(x_j) = f_{I_j}(y_j)$.
Let $\bar x$ be the $n$-tuple such that 
$\bar x_i = x_j$ if $i \in I_j$,
and similarly let $\bar y$ be the $n$-tuple such that 
$\bar y_i = y_j$ if $i \in I_j$. 
Then 
\begin{align*}
f(\bar x) = f( (x_1)^{(n)}_{I_1} %\cdot x^{(n)}_2 
\cdots (x_k)^{(n)}_{I_k}) & =
f((x_1)^{(n)}_{I_1}) 
%\cdot f((x_2)_{I_2}^{(n)}) 
\cdots f((x_k)^{(n)}_{I_k}) \\
& = f_{I_1}(x_1) \cdots f_{I_k}(x_k) \\
& = f_{I_1}(y_1) \cdots f_{I_k}(y_k)
= f( (y_1)^{(n)}_{I_1} %\cdot x^{(n)}_2 
\cdots (y_k)^{(n)}_{I_k}) = f(\bar y)
\end{align*}
showing that $f$ does not preserve $\neq$,
a contradiction. 
\end{proof}

\begin{proposition}\label{prop:monoid-pseudo-sig}
A monoid $(M;\cdot,1,\neq)$
has a pseudo-Siggers polymorphism 
 if and only if there are a polymorphism 
 $s \colon M^6 \to M$ and self-embeddings $a,b$ of $(M;\cdot,1)$ such that 
$$a(s_{\{i,j\}}(x)) = b(s_{\{i,j\}}(x))$$
for all $x \in M$ 
and $\{i,j\} \in \big \{ \{1,3\}, \{2,5\}, \{4,6\} \big \}$.
\end{proposition}
\begin{proof} 
For any
$x,y,z \in M$ we have 
$$a(s(x,y,x,z,y,z)) 
= a(s_{\{1,3\}}(x) s_{\{ 2,5\} }(y) s_{ \{4,6\} }(z)) 
= a(s_{\{1,3\}}(x)) a(s_{\{ 2,5\} }(y)) a(s_{ \{4,6\} }(z))$$
to which the result follows. 
\end{proof}

%Proposition~\ref{prop:monoid-decomp} implies for a pseudo-Siggers operation $s$ that
%$$s_{\{1,3\}}(x) s_{\{2,5\}}(y) s_{\{4,6\}} (z) = s(x,y,x,z,y,z) = s(u,v,u,w,v,w) = s_{\{1,3\}}(u) s_{\{2,5\}}(v) s_{\{4,6\}} (w)$$
%if and only if $$s_{\{1,6\}}(x) s_{\{2,4\}}(y) s_{\{3,5\}}(z) = s(y,x,z,x,z,y) = s(v,u,w,u,w,v) = s_{\{1,6\}}(u) s_{\{2,4\}}(v) s_{\{3,5\}}(w).$$
%Suppose that there exist $x,y,z \in M$ such that
%$f_{1,3}(x) = f_{2,5}(y) = f_{4,6}(z) = 1$. Then 
%$s(x,y,x,z,y,z) = 1$, and it follows that one of $x$, $y$, or $z$ equals $1$ because $s$ preserves
%the primitive positive definable set $\{x \in M \mid 
%x \neq 1\}$. 

\section{Groups}
\label{sect:groups}
Let $\bG$ be an $\omega$-categorical group. 
There are homogeneous 
$\omega$-categorical groups such that 
$\Csp(\bG,\neq)$ is undecidable (Section~\ref{sect:undec}). 
However, we are able to classify the complexity
of $\Csp(\bG,\neq)$ 
if $\bG$ is additionally abelian; in this case,
$\Csp(\bG,\neq)$ is in P or NP-complete (Section~\ref{sect:abelian}). 
Some of the structural results we 
obtain not only hold for abelian groups, 
but for general groups with a pseudo-Siggers polymorphism, and they will be presented in Section~\ref{sect:pseudo-sig-group}. 

The \emph{order} of an element $g \in G$
is the cardinality of the subgroup of $\bG$ 
generated by $g$. 
We say that $\bG$ is a \emph{torsion group} if every element of $\bG$ is of finite order. Since an $\omega$-categorical group is uniformly locally finite, it must be of   \emph{finite exponent}~\cite{Rosenstein73}, i.e., there exists $n \in \omega$ such that $g^n = 1$ for every $g \in G$; the minimum such $n$ is called the \emph{exponent} of $\bG$. This follows from the well-known fact that $\omega$-categorical 
structures $\bB$ 
are uniformly locally finite (Corollary~7.3.2 in~\cite{HodgesLong}), that is, there exists a function $f \colon \omega \to \omega$
such that for every $n \in \omega$ each substructure of $\bB$ generated by $n$ elements 
has at most $f(n)$ elements. In particular,
every $\omega$-categorical group must be a torsion group. 

\begin{remark}
Note that the identity element $1$ of $\bG$ has the quantifier-free definition $x \cdot x = x$ (in the language $\{\cdot\}$ of semigroups), and we may therefore assume that there is a constant symbol for $1$ in the signature. Similarly, the inverse function has a quantifier-free definition, 
and we assume that the signature contains a
unary function symbol for taking inverses. 
\end{remark}

\subsection{Pseudo-Siggers Groups}
\label{sect:pseudo-sig-group}
An \emph{involution} is an element of $\bG$ of order 2. 
An element $x \in G$ is \emph{central} if 
$x g = g x$ for all $g \in G$. Clearly, 
an involution generates a normal subgroup if
and only if it is central. 
Since every group is in particular a monoid, we may use Proposition~\ref{prop:monoid-decomp} and obtain the following. 

\begin{proposition}\label{prop:group-involution}
Let $\bG$ be an $\omega$-categorical group such that $(\bG,\neq)$ has a pseudo-Siggers polymorphism
$f$. 
Then 
\begin{itemize}
\item $\bG^2 \hookrightarrow \bG$, or 	
\item There is a central involution $i \in G$
such that $\bG \times \bG/\langle i \rangle \hookrightarrow \bG$. 
\end{itemize}
%Moreover,
%if $\bG$ is model-complete then $i$ has a primitive positive definition in $\bG$ and have reduction from G/<x> to G. DO WE NEED THIS? 
\end{proposition}
\begin{proof}
By Proposition~\ref{prop:monoid-decomp}
we may assume without loss of generality 
that $f_{\{1,3\}}$ is an embedding. 
Let $g$ be the binary polymorphism of 
$(\bG,\neq)$ given by 
$$g(x,y) := f_{\{1,3\}}(x) \cdot f_{\{2,4,5,6\}}(y) = f(x,1,x,1,1,1) \cdot f(1,y,1,y,y,y) = f(x,y,x,y,y,y).$$
{\bf Claim.} $\Im(g)$ is bi-embeddable with %$\Im(f_{\{1,3\}}(x) 
$\bG \times \Im(f_{\{2,4,5,6\}})$. 
Suppose that there exist $x,y \in G$ such that $g(x,1) = g(1,y)$. Then $x = 1$ or $y =1$ since $g$ preserves $\neq$. 
So $\Im(f_{\{1,3\}}) \cap \Im(f_{\{2,4,5,6\}}) = \{1\}$. Moreover, the subgroups 
$\Im(f_{\{1,3\}})$ and $\Im(f_{\{2,4,5,6\}})$ of $\bG$ are commuting, since for all $x,y \in G$
$$f_{\{1,3\}}(x)f_{\{2,4,5,6\}}(y) = f(x,y,x,y,y,y)=f_{\{2,4,5,6\}}(y) f_{\{1,3\}}(x).$$
It follows that $\Im(g)$ is generated by $\Im(f_{\{1,3\}}) \cup \Im(f_{\{2,4,5,6\}})$, and 
%is isomorphic to $\Im(f_{\{1,3\}})  \times \Im(f_{\{2,4,5,6\}}) = 
bi-embeddable with 
$\bG \times \Im(f_{\{2,4,5,6\}})$, as required. 

If $f_{\{2,4,5,6\}}$ is an embedding, then 
$g \colon \bG^2 \hookrightarrow \bG$ and we are done, so suppose that this is not the case. 
Then there exists an $i \in G \setminus \{1\}$ such that $f_{\{2,4,5,6\}}(i) = 1$. 
So we have $f(1,i,1,i,i,i) = f(1,1,1,1,1,1)$
and Equation~\ref{eq:siggers} implies that
$f(i,1,i,1,i,i) = f(1,1,1,1,1,1)$.
Hence, 
$f(i,i,i,i,i^2,i^2) = 1$, and we must have
$i^2=1$ because otherwise $f$ would 
not preserve $G \setminus \{1\}$. 
Note also that if $x \in G$ is such that 
$f(1,x,1,x,x,x,x) = 1$ then $x=i$ or $x=1$:
otherwise, $f(1,x,1,x,x,x,x)=1=f(i,1,i,1,i,i)$
in contradiction to the assumption that 
$f$ preserves $\neq$. So the kernel 
of $f_{\{2,4,5,6\}}$ is $\{1,i\} = \langle i \rangle$ and $(x,y \langle i \rangle) \mapsto g(x,y)$ is an 
embedding $\bG \times \bG/ \langle i \rangle \hookrightarrow \bG$. 
\end{proof}

% Extra part of Prop 6.8 uses 6.7 and 6.4. 
% but not yet clear whether this is needed at all. 
%\subsection{Products}

%\begin{lemma}% Lemma 6.3:
% Let $\bA$ be an $\omega$-categorical 
% algebra such that $\bA^2 \hookrightarrow \bA$.  
% A convex omega-categorical 
%algebra A is Siggers. Moreover, if B is a Siggers
%algebra of the same signature as A then $A \times  B$ is Siggers.
%\end{lemma}

%Lemma 6.4. Let A be an omega-categorical Siggers algebra. Then any p.p.-definable subalgebra
%B is Siggers. 

%Lemma 6.7. 

Let $\bH$ and $\bK$ be $\omega$-categorical 
groups. We say that $\bH$ and $\bK$ are \emph{of relatively prime exponent} if the exponents of $\bH$ 
and $\bK$ are co-prime.

\begin{lemma}\label{lem:group-product}
%Cor 6.9: 
Let $\bH$ and $\bK$ be $\omega$-categorical 
groups of relatively prime exponent. If the structure 
$(\bH \times \bK,\neq)$ has a pseudo-Siggers polymorphism, then 
$(\bH,\neq)$ and $(\bK,\neq)$ have pseudo-Siggers polymorphisms, too. 
\end{lemma}
\begin{proof}
Let $h \in \omega$ be the exponent of $\bH$
and $k \in \omega$ the exponent of $\bK$,
and let $\bG := \bH \times \bK$.
 Suppose   that $(\bG,\neq)$ has a pseudo-Siggers polymorphism $s$. Then
 $\bH$ is isomorphic to the subgroup $\bH \times \{1^{\bK}\}$ of $\bG$, which 
 is precisely the set of all elements of $G$
 that satisfy $x^h = 1$;
 hence, $H \times \{1^{\bK}\}$ is primitive positive definable in $\bG$, and the restriction of $s$ to this set is a pseudo-Siggers polymorphism of $(\bH \times \{1^{\bK}\},\neq)$. Therefore, $(\bH,\neq)$ has a pseudo-Siggers polymorphism. Similarly, $(\bK,\neq)$ has a pseudo-Siggers polymorphism. 
% The converse direction holds for 
%identities in polymorphism clones of $\omega$-categorical structures in general. 
 %Conversely, suppose that both $\bH$ and $\bK$ has pseudo-Siggers polymorphisms. 
% Since they are of relatively prime exponent, we may assume without loss of generality that $h$ is odd, and thus $\bH$ contains no involutions. By Proposition~\ref{prop:group-involution}, $\bH$ must be convex. 
\end{proof} 
%uses 6.4 and 6.8 and 6.3.
%6.10 (homogeneous)

\subsection{Undecidable $\omega$-categorical
groups}
\label{sect:undec}
Saracino and Wood~\cite{SaracinoWood}
showed that there are $2^\omega$ non-isomorphic homogeneous $\omega$-categorical groups. Homogeneous $\omega$-categorical structures have quantifier elimination~\cite{Hodges} and in particular they
are model-complete. Hence, if two homogeneous $\omega$-categorical structures are companions, then they must be isomorphic~\cite{Hodges}. 
Recall from Lemma~\ref{lem:c} that $\omega$-categorical algebras $\bA$ and $\bB$ are companions if and only if $\Csp(\bA,\neq)$ and $\Csp(\bB,\neq)$ are the same computational problem. Since there are only countably many Turing machines, it follows
that there are $\omega$-categorical homogeneous  groups $\bA$ such that $\Csp(\bA,\neq)$ is undecidable. 

\subsection{The abelian case}
We now consider $\omega$-categorical abelian groups $\bG$. Our main results
are a characterisation of the case that $(\bG,\neq)$ has a pseudo-Siggers polymorphism 
and a full complexity classification for $\Csp(\bG,\neq)$.
 It is standard to then use additive notation;
so the identity element is from now on denoted by $0$ and the group composition by $+$. 

Recall that 
a \emph{$p$-group} is a group whose elements have orders that are powers of a fixed prime $p$. 
For example, the cyclic group ${\mathbb Z}_{p^n}$ of order $p^n$ is a $p$-group. 
A subgroup of $\bG$ which is a $p$-group
is also called a \emph{$p$-subgroup}. 
We write $\bigoplus_{i \in I} G_i$ for the \emph{direct sum} of the $G_i$, i.e., for
the subgroup of $\prod_{i \in I} G_i$ 
containing all elements that are $1$ at all but finitely many indices. 
If $G_i = G$ for all $i \in I$ then 
we also write $G^{(I)}$ 
instead of $\bigoplus_{i \in I} G$. 
Finite direct products $\bG_1 \times \cdots \times \bG_k$
coincide with finite direct sums $\bG_1 \oplus \cdots \oplus \bG_k$ 
and we use the latter notation in this section. 
As mentioned earlier, an $\omega$-categorical group $\bG$ must be a torsion group.

%\begin{theorem}[see~\cite{AbelianGroups}, Theorem 8.5]
%Every countable elementary $p$-group is 
%of the form ${\mathbb Z}^{(\omega)}_p$. 
%\end{theorem}
%A proof of the following theorem can be found in ~\cite{AbelianGroups} (Theorem 8.4). 

\begin{theorem}[Theorem 1 in~\cite{Kaplansky}]
\label{thm:class-p}
Every abelian torsion group $\bG$ is the direct sum of its $p$-subgroups. 
\end{theorem}

Recall that every $\omega$-categorical group
must be of finite exponent. 

\begin{corollary}
\label{cor:class-p}
Every abelian group $\bG$ of finite exponent is a finite direct sum of its $p$-subgroups. 
\end{corollary}

As a consequence of Corollary~\ref{cor:class-p} 
and Lemma~\ref{lem:group-product} we
need to consider abelian $p$-groups. 
We first recall another basic fact from 
the theory of abelian groups. 

\begin{theorem}[Theorem 6 in~\cite{Kaplansky}]\label{thm:class-cyclic}
Every abelian group of finite exponent is a direct sum of cyclic groups. 
\end{theorem}

A group is called \emph{trivial} if it only consists of the identity element $1$, and \emph{non-trivial} otherwise.

\begin{corollary}\label{cor:nf}
Let $\bG$ be a non-trivial 
countable $\omega$-categorical
$p$-group. Then there exists $k \in \omega$ 
such that 
\begin{align}
\bG = {\mathbb Z}_{p^{n_1}}^{(s_1)} \oplus
\cdots \oplus {\mathbb Z}_{p^{n_k}}^{(s_k)}
\label{eq:omega-cat-p-group}
\end{align}
where $n_1,\dots,n_k \in \{1,2,3,\dots\}$ and 
$s_1,\dots,s_k \in \{1,2,\dots\} \cup \{\omega\}$.  
\end{corollary}
\begin{proof}
%If $\bG$ is trivial then 
Theorem~\ref{thm:class-cyclic} shows 
that
$\bG = \bigoplus_{i \in \omega} {\mathbb Z}_{c_i}$ for some integer sequence $(c_i)_{i \in \omega}$ with $c_i \geq 2$ for all $i \in \omega$. 
Theorem~\ref{thm:class-p} implies
that each $c_i$ is of the form $c_i = p^{n_i}$ for
some $n_i \in \{1,2,3,\dots\}$. 
So we may write 
$\bG = \bigoplus_{i \in \omega} 
{\mathbb Z}_{p^{n_i}}^{(s_i)}$ where
$s_i \in \omega \cup \{\omega\}$.  
As $\bG$ has finite exponent we have $s_i=0$ for all but
finitely many $i \in \omega$.
%, because
%group elements of different order lie in different orbits
%of $\Aut(\bG)$, 
%and automorphism groups of
%$\omega$-categorical structures
%have only finitely many orbits~\cite{Hodges}. 
\end{proof}

Understanding $p$-groups
up to bi-embeddability will be useful. 
For groups of the form (\ref{eq:omega-cat-p-group}) we define 
$$m_{\bG} := \begin{cases} 0 &  \text{ if } s_i \in \omega \text{ for every } i \in \{1,\dots,k\} \\
\max \{n:s_n=\omega\} & \text{ otherwise.} 
\end{cases}
$$
Note that $\bG$ is finite if and only if $m_{\bG} = 0$. 
The following is a consequence of a more
general result about bi-embeddability of abelian
$p$-groups~\cite{BarwiseEklof} (Corollary 5.4); see Remark 4.12 in~\cite{CalderoniThomas}.

% Lemma 6.12: 
\begin{lemma}\label{lem:bi-embed}
Let $k,\ell \in \omega$ and $s_n, t_n \in \omega \cup \{\omega\}$. 
Then 
$\bG = \bigoplus_{n \in \{1,\dots,k\}} {\mathbb Z}^{(s_n)}_{p^n}$
and 
$\bH = \bigoplus_{n \in \{1,\dots,\ell\}} {\mathbb Z}^{(t_n)}_{p^n}$ are bi-embeddable if and only if
%$\bG$ and $\bH$ are isomorphic finite $p$-groups or 
$m_{\bG} = m_{\bH}$ and $s_n = t_n$ for all $n \geq m_{\bG}$.  
\end{lemma}

% Cor 6.14:
We apply this lemma to the two possibilities that
arise in Proposition~\ref{prop:group-involution}
and start with the easier situation where 
$\bG^2 \hookrightarrow \bG$. 

\begin{lemma} 
\label{lem:square}
Let $\bG$ be an $\omega$-categorical abelian $p$-group. Then $\bG^2 \hookrightarrow \bG$ if and only if $\bG$ is
bi-embeddable with ${\mathbb Z}^{(\omega)}_{p^n}$ for some $n \in \omega$. 
%$n \in \{0,1,2,\dots\}$. 
\end{lemma}
\begin{proof}
The statement is trivial if $\bG$ is trivial. Otherwise, by Corollary~\ref{cor:nf}, $\bG$ can be written
as $\bG = \bigoplus_{n \in \{1,\dots,k\}} {\mathbb Z}_{p^n}^{(s_n)}$. Then $s_n = 2s_n$ for every $n > m_{\bG}$ by Lemma~\ref{lem:bi-embed}. 
Hence, $s_n = 0$ for every $n \geq m_{\bG}$, 
and again by Lemma 6.14 we conclude that $\bG$ is bi-embeddable with ${\mathbb Z}^{(\omega)}_{p^{m_{\bG}}}$. The converse is immediate. 
\end{proof}

We now treat the other possibility that arises in Proposition~\ref{prop:group-involution} and which involves the quotient 
$\bG/\langle x \rangle$ for 
a central 
involution $x$ of $\bG$. We first have to recall how this quotient looks like if $\bG$ is a countable abelian $\omega$-categorical $2$-group and hence of the form as described in Corollary~\ref{cor:nf}. 

\begin{lemma}\label{lem:quotient}
Let $\bG = {\mathbb Z}_{2^1}^{(s_1)} \oplus \cdots \oplus  {\mathbb Z}_{2^k}^{(s_k)}$
be an abelian $\omega$-categorical
2-group and let $x \in G$ be a central involution.
Then $G/\langle x \rangle$ is isomorphic
to 
$${\mathbb Z}_{2^1}^{(s_1)} \oplus \cdots \oplus {\mathbb Z}_{2^{i-2}}^{(s_{i-2})} \oplus 
{\mathbb Z}_{2^{i-1}}^{(s_{i-1}+1)} \oplus
{\mathbb Z}_{2^{i}}^{(s_{i}-1)} \oplus
{\mathbb Z}_{2^{i+1}}^{(s_{i+1})} \oplus \cdots \oplus {\mathbb Z}_{2^k}^{(s_k)}.$$
%$${\mathbb Z}_{2^1}^{(s_1)} \oplus \cdots \oplus  {\mathbb Z}_{2^k}^{(s_k)}.$$
for some $i \in \{1,\dots,k\}$ (where $s_0=0$ if $i=1$). 
\end{lemma}
\begin{proof} 
Let $g_1,g_2,\dots \in G$ be such that $\bG = \bigoplus_{i \in \omega} \langle g_i \rangle$. 

{\bf Claim.}  $x$ is a power of some generator $g_i$. Suppose without loss of generality
that $x = (m_1 g_1,\dots,m_r g_r,0,\dots)$ for $r \in \omega$ and $m_1,\dots,m_r \in \omega$ powers of two. Let $t \in \{1,\dots,r\}$ be such that $m_t$ is minimal among $m_1,\dots,m_r$, 
and let $a_i := m_i/m_t$ for $i \in \{1,\dots,r\}$. 
Let $y = (a_1 g_1,\dots,a_r g_r,0,0,\dots)$
be so that $x = y^{m_t}$. 
%Let $\bH$ be the subgroup of $\bG$ generated 
%by $g_1,\dots,g_r$. 
Since $a_1,\dots,a_r$ have greatest common divisor $1$, 
we may use the following lemma (Lemma II.3.b in~\cite{Schenkman}): if $g_1,\dots,g_r$ are generators of an abelian group $\bH$, and if $a_1,\dots,a_r$ are integers with greatest common divisor $1$, 
then the element
$a_1 g_1 + a_2 g_2 + \cdots + a_r g_r$ is one of a set of $r$ generators of $\bH$. 
So there are $h_2,\dots,h_r \in G$
such that $\bG$ can be written as
$$\bG = \langle y \rangle \oplus \bigoplus_{i \in \{2,\dots,r\}} \langle h_i \rangle \oplus \bigoplus_{i \in \{r,r+1,\dots\}} \langle g_i \rangle$$
Since ${\mathbb Z}_{2^s} / {\mathbb Z}_2$ is isomorphic to ${\mathbb Z}_{2^{s-1}}$,
we get that %$\big(\bigoplus_{i \in \omega} \langle g_i \rangle\big)
${\mathbb Z}_{2^1}^{(s_1)} \oplus \cdots \oplus {\mathbb Z}^{(s_k)}_{2^k}/\langle x \rangle$ 
is of the form as described in the statement. 
\end{proof}

\begin{lemma}\label{lem:2}
A non-trivial $\omega$-categorical abelian 2-group $\bG$ is bi-embeddable with $\bG \oplus \bG/\langle x\rangle$ 
for some involution $x \in G$ if and only if 
\begin{itemize}
\item $\bG^2 \hookrightarrow \bG$, or 
\item $\bG$ is bi-embeddable with ${\mathbb Z}^{(\omega)}_{2^n} \oplus {\mathbb Z}_{2^{n+1}}$ for some $n \in \omega$. 
\end{itemize}
\end{lemma}
% uses 6.12
\begin{proof}
By Corollary~\ref{cor:nf}, the 2-group 
$\bG$ can be written as 
%${\mathbb Z}_2^{(n_1)} \oplus \cdots \oplus {\mathbb Z}_{2^n}^{(n_k)} quatsch
$\bigoplus_{n \in \{1,\dots,k\}} {\mathbb Z}_{2^n}^{(s_n)}$ for $s_k > 0$. 
First suppose that 
$\bG$ is bi-embeddable with $\bG \oplus \bG/\langle x\rangle$ 
for some involution $x \in G$. 
If $\bG^2 \hookrightarrow \bG$ then we are done. Otherwise, 
by Lemma~\ref{lem:square} the group $\bG$ is 
not bi-embeddable with ${\mathbb Z}^{(\omega)}_{p^n}$ for some $n \in \omega$. 
Corollary~\ref{cor:nf} then implies that $s_k < \omega$. By Lemma~\ref{lem:quotient} there exists a unique $i \in \{1,\dots,k\}$ such that $G/\langle x \rangle$ is isomorphic to 
$${\mathbb Z}_2^{(s_1)} \oplus \cdots \oplus {\mathbb Z}_{2^{i-2}}^{(s_{i-2})} \oplus 
{\mathbb Z}_{2^{i-1}}^{(s_{i-1}+1)} \oplus
{\mathbb Z}_{2^{i}}^{(s_{i}-1)} \oplus
{\mathbb Z}_{2^{i+1}}^{(s_{i+1})} \oplus \cdots \oplus {\mathbb Z}_{2^k}^{(s_k)}.$$
Since $\bG$ is bi-embeddable with 
$\bG \oplus \bG / \langle x \rangle$ 
Lemma~\ref{lem:bi-embed} implies that either $s_k = s_k (s_k - 1)$ if $k=i$
or $s_k = 2s_k$ otherwise. 
Since $s_k$ is finite and non-zero (as $\bG$ contains an involution, it must be non-trivial) we conclude that $k=i$ and $s_k = 1$. 
If $s_{k-1}$ is finite then 
$s_{k-1} = 2s_{k-1} + 1$
by Lemma~\ref{lem:bi-embed}, a contradiction.
So $s_{k-1} = \omega$, and by Lemma~\ref{lem:bi-embed} 
$\bG$ is bi-embeddable with ${\mathbb Z}^{(\omega)}_{2^{k-1}} \oplus {\mathbb Z}_{2^k}$, as required. 

Conversely, if $\bG^2 \hookrightarrow \bG$
then by Lemma~\ref{lem:square}
$\bG$ is bi-embeddable with 
${\mathbb Z}^{(\omega)}_{2^n}$ 
for some $n \in \omega$, and in fact $n \geq 1$ since $\bG$ is non-trivial. 
Hence, there exists an involution $x \in G$ and 
$\bG/\langle x \rangle$ is biembeddable with ${\mathbb Z}_{2^{n-1}} \oplus {\mathbb Z}_{2^n}^{(\omega)}$ by Lemma~\ref{lem:quotient}, which is bi-embeddable with ${\mathbb Z}_{2^n}^{(\omega)}$ by Lemma~\ref{lem:bi-embed}. 
We conclude that $\bG$ is bi-embeddable with 
$\bG \oplus \bG/\langle x \rangle$. 
If $\bG$ is bi-embeddable with ${\mathbb Z}^{(\omega)}_{2^n} \oplus {\mathbb Z}_{2^{n+1}}$ for some $n \in \omega$, let $x$ be an involution generated by an element
of order $2^{n+1}$. Then $\bG/\langle x \rangle$ is bi-embeddable with $ 
 {\mathbb Z}^{(\omega)}_{2^n} \oplus {\mathbb Z}_{2^{n}}$, which is isomorphic to ${\mathbb Z}^{(\omega)}_{2^n}$ 
and hence $\bG \oplus \bG/\langle x \rangle$ is bi-embeddable with ${\mathbb Z}^{(\omega)}_{2^n} \oplus {\mathbb Z}_{2^{n+1}} \oplus {\mathbb Z}^{(\omega)}_{2^n}$
and hence with $\bG$. 
%${\mathbb Z}^{(\omega)}_{2^n} \oplus {\mathbb Z}_{2^{n+1}}$. 
\end{proof}

\begin{proposition}\label{prop:pseudo-sig-abelian-group}
Let $\bG$ be an $\omega$-categorical abelian group such that $(\bG,\neq)$ has a pseudo-Siggers polymorphism. Then $\bG$ is bi-embeddable with 
${\mathbb Z}_m^{(\omega)}$ or with 
${\mathbb Z}_m^{(\omega)} \oplus {\mathbb Z}_{2m}$ for some $m \geq 1$. 
\end{proposition}
\begin{proof}
If $\bG$ is trivial then it is bi-embeddable with 
${\mathbb Z}_1^{(\omega)}$ and we are done. 
Otherwise, by Corollary~\ref{cor:class-p} the group $\bG$
is $\bG_{p_1} \oplus \cdots \oplus \bG_{p_r}$,  where $r \geq 1$, 
$p_1,\dots,p_r$ are primes,  and $\bG_{p_i}$ for $i \in \{1,\dots,r\}$ is 
a non-trivial $p_i$-subgroup of $\bG$. By Lemma~\ref{lem:group-product}, 
for each $p \in \{p_1,\dots,p_r\}$
the structure $(\bG_p,\neq)$ has a pseudo-Siggers polymorphism. 

By Proposition~\ref{prop:group-involution} we have 
$\bG_p^2 \hookrightarrow \bG_p$ 
or $\bG_p$ is bi-embeddable with $\bG_p \oplus \bG_p/\langle x \rangle$ for some involution $x \in G$ and $p=2$. In the first case, $\bG_p$ is 
bi-embeddable with ${\mathbb Z}^{(\omega)}_{p^n}$ for some $n \in \omega$ by Lemma~\ref{lem:square}. 
In the latter case, $\bG_p$ is bi-embeddable with ${\mathbb Z}_{2^n}^{(\omega)} \oplus {\mathbb Z}_{2^{n+1}}$ or with 
${\mathbb Z}_{2^n}^{(\omega)}$
for some $n \in \omega$ by Lemma~\ref{lem:2}. 
So we deduce that there are $n_1,\dots,n_r \in \{1,2,3,\dots\}$ such that $\bG$ is bi-embeddable with 
\begin{itemize}
\item 
${\mathbb Z}^{(\omega)}_{{p_1}^{n_1}} \oplus \cdots \oplus {\mathbb Z}^{(\omega)}_{{p_r}^{n_r}}$ or with 
%\item
%${\mathbb Z}_2 \oplus {\mathbb Z}^{(\omega)}_{{p_1}^{n_1}} \oplus \cdots \oplus {\mathbb Z}^{(\omega)}_{{p_r}^{n_r}}$ where $p_i > 2$ for every $i \in \{1,\dots,r\}$, or
\item ${\mathbb Z}^{(\omega)}_{p_1^{n_1}} \oplus {\mathbb Z}_{p_1^{n_1+1}} \oplus {\mathbb Z}^{(\omega)}_{{p_2}^{n_2}} \oplus \cdots \oplus 
{\mathbb Z}^{(\omega)}_{{p_r}^{n_r}}$ 
where $p_1 = 2$ and $p_i > 2$ for every $i \in \{2,\dots,r\}$. 
\end{itemize}
Let $m := p_1^{n_1} p_2^{n_2} \cdots p_r^{n_r}$.
In case (1), the group $\bG$ is isomorphic to  
${\mathbb Z}_m^{(\omega)}$. 
In case (2), we have 
%let $n := p_2^{n_2} \cdots p_r^{n_r}$. 
\begin{align*}
{\mathbb Z}^{(\omega)}_{p_1^{n_1}} \oplus {\mathbb Z}_{p_1^{n_1+1}} \oplus {\mathbb Z}^{(\omega)}_{{p_2}^{n_2}} \oplus \cdots \oplus 
{\mathbb Z}^{(\omega)}_{{p_r}^{n_r}}
& \simeq 
 {\mathbb Z}^{(\omega)}_{{p_1}^{n_1}} \oplus \cdots \oplus 
{\mathbb Z}^{(\omega)}_{{p_r}^{n_r}}
\oplus {\mathbb Z}_{2^{n_1+1}} \oplus {\mathbb Z}_{{p_2}^{n_2}} \oplus \cdots \oplus 
{\mathbb Z}_{{p_r}^{n_r}}  \\
& \simeq 
{\mathbb Z}_m^{(\omega)} \oplus {\mathbb Z}_{2m}.\end{align*}
\end{proof}

\subsection{Polynomial-time tractable abelian groups}
Let $n \in \omega$. 
In this section we present a polynomial-time
algorithm for $\Csp({\mathbb Z}_{2n} \oplus {\mathbb Z}_n^{(\omega)},\neq)$. Continuing with the additive
notation, we let ${\mathbb Z}_k = \{0,1,\dots,k-1\}$ for each $k\in \omega$. So let $\Phi$ be a conjunction of atomic formulas of the form $x=y+z$ 
and of the form $x \neq y$ over a finite set of variables $V$. 
We would like to test whether $\Phi$ is satisfiable in ${\mathbb Z}_{2n} \oplus {\mathbb Z}_n^{(\omega)}$
for some fixed $n \geq 2$. 

\begin{remark}
Note that over ${\mathbb Z}_2$ every disequality $x \neq y$ can be translated into an equality $x = y+1$ and hence satisfiability of the entire system can be solved in polynomial time with Gaussian elimination. 
The same trick does not work for 
solvability in ${\mathbb Z}_n$ if $n \geq 3$, and indeed satisfiability of disequalities over ${\mathbb Z}_n$ is NP-complete for $n \geq 3$. 
\end{remark}

Linear equation systems over
${\mathbb Z}_{k}$, for any $k \in \omega$, can 
be solved in polynomial time~\cite{GoldmannRussell}. 
We need this algorithm for equation systems
over ${\mathbb Z}_{2n}$.
Alternatively, we can use a more general
algorithm of Bulatov and Dalmau for constraints preserved by a Maltsev operation~\cite{Maltsev}. 
%Observe that $({\mathbb Z}_{2n};+)$ has the Maltsev polymorphism 
Let $m \colon {\mathbb Z}_{2n}^3 \to {\mathbb Z}_{2n}$ be the Maltsev operation given by $(x,y,z) \mapsto x-y+z$. Observe that $m$ is idempotent, preserves the graph of addition, and also preserves the following relation 
$$R := \{(x_1,x_2) \in {\mathbb Z}^2_{2n} \mid x_1-x_2 = n\}.$$ To see this, let $(x_1,x_2),(y_1,y_2),(z_1,z_2) \in R$. Then $$m(x_1,y_1,z_1) - m(x_2,y_2,z_2) = x_1-x_2 - (y_1 - y_2) + z_1-z_2 = n-n+n=n.$$

\subsection*{The algorithm}
Let $\Phi_e$ be all the conjuncts in $\Phi$ that are equations, and let $\Phi_d$ be all the conjuncts in $\Phi$ that are disequalities. 
Our algorithm is the following.

\begin{enumerate}
\item Test for each disequality $x \neq y$ in $\Phi_d$ 
with Gaussian elimination whether $\Phi_e$ implies $x=y$ in ${\mathbb Z}_n$. 
Let $\Phi^*_d$ be the set of all disequalities where this is the case. 
\item For each inequality $x \neq y$ in $\Phi^*_d$,
add the constraint $R(x,y)$ to $\Phi_e$, and solve
the resulting instance of the CSP over ${\mathbb Z}_{2n}$ with the Bulatov-Dalmau algorithm for
Maltsev constraints. The algorithm accepts
if and only if the Maltsev instance is satisfiable. 
\end{enumerate} 

\begin{theorem}\label{thm:alg}
Let $n \geq 1$. 
Then the algorithm presented above solves
$\Csp({\mathbb Z}_{2n} \oplus {\mathbb Z}_n^{(\omega)},\neq)$ in polynomial time. 
\end{theorem}
\begin{proof}
It is clear that the algorithm has a polynomial running time. 
To prove the correctness of this algorithm, first suppose that $\Phi$ has a solution $s \colon V \to  {\mathbb Z}_{2n} \oplus {\mathbb Z}^{(\omega)}_n$. Let $x \neq y$ be a disequality
from $\Phi^*_d$, and let $s(x) = (s_1(x),s_2(x),\dots)$ and $s(y) = (s_1(y),s_2(y),\dots)$. 
By the definition of $\Phi^*_d$, 
we must have that $s(x)_i = s(y)_i$ for $i \in \{1,2,\dots\}$
and $s(x)_1 = s(y)_1 \; (\text{mod } n)$. 
%So $s(x)$ is
%of the form $(s(x)_1,0,0,\dots)$ and
%$s(y)$ is of the form $(s(y)_1,0,0,\dots)$ for 
%$s(x)_1,s(y)_1 \in {\mathbb Z}_{2n}$.
Since $s(x) \neq s(y)$ we must
have $s(x)_1 \neq s(y)_1$, and hence $s(x)_1 - s(y)_1 = n$ and 
$(s(x)_1,s(y)_1) \in R$. Therefore, if the Bulatov-Dalmau algorithm rejects, then our algorithm correctly rejects the input.
 
Conversely, suppose that the algorithm accepts. 
Hence, the input to the Maltsev constraints 
has a solution $r \colon V \to {\mathbb Z}_{2n}$. 
Let $\phi_1,\dots,\phi_k$ be the disequalities in $\Phi_d \setminus \Phi_d^*$.
For each $i \leq k$ there exists a solution
$t_i$ 
to $\Phi_e$ over ${\mathbb Z}_n$ such that 
$t_i$ satisfies $\phi_i$. 
We then construct a solution
$s \colon V \to {\mathbb Z}_{2n} \oplus {\mathbb Z}^k_n$ of $\Phi$ as follows: 
$$s(x) := (r(x),t_1(x),\dots,t_k(x))$$
which may naturally be viewed as a solution 
$s \colon V \to {\mathbb Z}_{2n} \oplus {\mathbb Z}^{(\omega)}_n$. 
The map $s$ satisfies $\Phi_e$ since each of 
$r,t_1,\dots,t_k$ does. 
Moreover, $s$ satisfies $\Phi_d^*$ since $r$ does.
Finally, $s$ satisfies $\phi_i$ for $i \in \{1,\dots,k\}$
since $t_i$ does.
\end{proof}

\subsection{The classification}
\label{sect:abelian}
We combine the results obtained in the previous sections to prove our complexity dichotomy for $\omega$-categorical abelian groups (Theorem~\ref{thm:abelian-groups}).  
The border is given by the existence
of a pseudo-Siggers polymorphism of the model companion of $(\bG,\neq)$. 
Then we strengthen the statement by
providing an exact characterisation of those
$\omega$-categorical abelian groups $\bG$
such that $(\bG,\neq)$
has a pseudo-Siggers polymorphism (Theorem~\ref{thm:group-pseudo-sig}). 
We finally prove that for abelian groups 
the two cases 
in Theorem~\ref{thm:BP} are disjoint, which provides yet another equivalent characterisation 
of the complexity border in terms of uniformly 
continuous minor-preserving maps to $\Proj$. 

% (without first passing to the model companion of $\bG$).  

%a strengthened version 
%that shows that the complexity border is precisely given by the existence of a pseudo-Siggers polymorphism. 

\begin{proposition}\label{prop:NP}
If $\bG$ is an $\omega$-categorical abelian group then
$\Csp(\bG,\neq)$ is in NP. 
\end{proposition}
\begin{proof} By Theorem~\ref{thm:class-p} the group  $\bG$ can be written as $\bG_{p_1}\oplus \cdots \oplus \bG_{p_r}$, where $r\in \omega$ and $p_1,\dots,p_r$ are distinct primes, and $\bG_{p_i}$ for $i\in \{1,\dots,r\}$ is a non-trivial $p_i$-subgroup of $\bG$. 
By Lemma \ref{lem:bi-embed} each $\bG_{p_i}$ is  bi-embeddable with $\mathbb{Z}_{p_i^{n_i}}^{(\omega)} \oplus \bH_i$ for some $n_i\in \omega\cup \{0\}$ and finite abelian group $\bH_i$.
 Hence $\bG$ is bi-embeddable with $\mathbb{Z}_n^{(\omega)}\oplus \bH$, where $n= p_1^{n_1}p_2^{n_2}\cdots p_r^{n_r}$  and $\bH=\bH_1\oplus \bH_2 \oplus \cdots \oplus \bH_r$.
  Since $\Csp(\mathbb{Z}_n^{(\omega)},\neq)$ is in P and $(\mathbb{Z}_n^{(\omega)})^2 \hookrightarrow \mathbb{Z}_n^{(\omega)}$, it follows from Proposition \ref{prop:newproduct} that there is a polynomial-time reduction from $\Csp(\mathbb{Z}_n^{(\omega)}\oplus \bH,\neq)$ to $\Csp(\bH,\neq)$, and hence from $\Csp(\bG,\neq)$ to $\Csp(\bH,\neq)$. Since $\bH$ is finite, the result follows. 
\end{proof}

\begin{theorem}\label{thm:abelian-groups}
Let $\bG$ be an $\omega$-categorical abelian group and let $\bH$ be its model companion.
If $(\bH,\neq)$ has a pseudo-Siggers polymorphism, then 
%$\Csp(\bH;\neq)$ and 
$\Csp(\bG,\neq)$ is in P.
Otherwise, $\Csp(\bG,\neq)$ is NP-complete. 
\end{theorem}
\begin{proof}
If $(\bH,\neq)$ does not have a pseudo-Siggers polymorphism, then $\Csp(\bH,\neq)$ and therefore also $\Csp(\bG,\neq)$ are NP-hard by Theorem~\ref{thm:BP}, and thus NP-complete by
Proposition~\ref{prop:NP}. 
Otherwise, Proposition~\ref{prop:pseudo-sig-abelian-group} implies that $\bH$ is bi-embeddable with ${\mathbb Z}_n^{(\omega)}$
or with ${\mathbb Z}_n^{(\omega)} \oplus {\mathbb Z}_{2n}$ for some $n \in \omega$. 
In this case the polynomial-time tractability of
$\Csp(\bH,\neq)$ and therefore also of $\Csp(\bG,\neq)$ follows from Theorem~\ref{thm:alg}. 
\end{proof}

%We now characterise those $\omega$-categorical
%abelian groups 
%(not necessarily model-complete)
%that have a pseudo-Siggers polymorphism. 

The border between polynomial-time tractable and NP-hard cases can be described mathematically in several equivalent ways. 

\begin{theorem}\label{thm:group-pseudo-sig}
Let $\bG$ be an $\omega$-categorical abelian group. Then the following are equivalent. 
\begin{enumerate}
\item $(\bG,\neq)$ has a pseudo-Siggers polymorphism; 
\item $\bG$ is bi-embeddable with either ${\mathbb Z}_n^{(\omega)}$ or ${\mathbb Z}_n^{(\omega)} \oplus {\mathbb Z}_{2n}$ for some $n \geq 1$; 
\item the model-complete core of $(\bG,\neq)$ has a pseudo-Siggers polymorphism. 
\end{enumerate}
\end{theorem}
\begin{proof}
The implication from $(1)$ to $(2)$ is Proposition~\ref{prop:pseudo-sig-abelian-group}.

To prove the implication from $(2)$ to $(3)$
it suffices to prove that ${\mathbb Z}_n^{(\omega)}$ and ${\mathbb Z}_n^{(\omega)} \oplus {\mathbb Z}_{2n}$ are model-complete and have a pseudo-Siggers polymorphism. Since $({\mathbb Z}_n^{(\omega)})^2$ is isomorphic to ${\mathbb Z}_n^{(\omega)}$, the structure $({\mathbb Z}_n^{(\omega)},\neq)$ has a pseudo-Siggers polymorphism by Lemma~\ref{lem:square-pseudo-sig}. Moreover, it is well-known and easy to see that ${\mathbb Z}_n^{(\omega)}$ is homogeneous, and therefore model-complete. 

Now let $\bG = {\mathbb Z}_n^{(\omega)} \oplus {\mathbb Z}_{2n}$. Let $(a_i)_{i \in \omega}$ be a sequence of elements of $G$ of order $n$ and let $b$ be an element of $G$ of order $2n$ such that $\bG = \bigoplus_{i \in \omega} \langle a_i \rangle \oplus \langle b \rangle$. We construct a pseudo-Siggers polymorphism $s$ of $\bG$ as follows. 
Let $g \colon \big (\bigoplus_{i \in \omega} \langle a_i \rangle \big)^6 \to \bigoplus_{i \in \omega \setminus \{1,\dots,6\}} \langle a_i \rangle$ be an isomorphism. 
Note that the map 
$h \colon \big (\bigoplus_{i \in \omega} \langle a_i \rangle \big)^6 \to \bigoplus_{i \in \omega \setminus \{1,\dots,6\}} \langle a_i \rangle$ given by $h(x_1,x_2,x_3,x_4,x_5,x_6) := g(x_2,x_1,x_4,x_3,x_6,x_5)$ is an isomorphism, too, and hence there exists an automorphism
$\alpha$ of $\bigoplus_{i \in \omega \setminus \{1,\dots,6\}} \langle a_i \rangle$
such that $\alpha(h(x_1,\dots,x_6)) = g(x_1,\dots,x_6)$. Note that for all 
$x,y,z \in \bigoplus_{i \in \omega} \langle a_i \rangle$
we have 
$$g(x,y,x,z,y,z) = \alpha(h(x,y,x,z,y,z)) = \alpha(g(y,x,z,x,z,y)).$$
Extend $g$ to a homomorphism $f \colon \bG^6 \to \bG$ by defining 
$$f_{\{k\}}(b) := \begin{cases} a_k + b & \text{ if } k \in \{1,4,5\} \\
a_k + 2b & \text{ if } k \in \{2,3,6\}
\end{cases}.$$ 
This fully determines $f$ by (\ref{eq:decomp})
and 
because $\bG$ is generated by $b$ and $(a_i)_{i \in \omega}$. 
We claim that $f$ preserves $\neq$. Indeed, 
suppose that
$f(x_1,\dots,x_6) = f(y_1,\dots,y_6)$.
Write $x_i$ as $(c_i,r_ib)$ and 
$y_i$ as $(d_i,s_ib)$ where $c_i$ and $d_i$  are elements of $\bigoplus_{i\in \omega} \langle a_i \rangle$ and $r_i,s_i \in \{0,\dots,2n-1\}$. 
From the definition of $f$ we then obtain that
$g(c_1,\dots,c_6)=g(d_1,\dots,d_6)$, which implies that $c_1=d_1,\dots,c_6=d_6$ since $g$ is injective. So $f(r_1b,\dots,r_6b) = f(s_1b,\dots,s_6b)$, and it suffices to show that
$r_i b = s_i b$ for some $i \in \{1,\dots,6\}$. 
Let $r := r_1+2r_2+2r_3+r_4+r_5+2r_6$
and $s := s_1+2s_2+2s_3+s_4+s_5+2s_6$. 
Then 
\begin{align*}
f(r_1b,\dots,r_6b) =  (r_1a_1 + \cdots + r_6a_6) + rb = (r_1a_1 + \cdots + r_6a_6) + sb =  f(s_1b,\dots,s_6b)
\end{align*}
forces $r-s = 0 \mod 2n$ and $r_i = s_i \mod n$ for every $i \in \{1,\dots,6\}$, say $r_i=s_i+k_in$. Then as $2r_i=2s_i \mod 2n$ we have 
$$ r-(2r_2+2r_3+2r_6) = r_1+r_4+r_5 = s_1 +s_4+s_5 = s-(2s_2+2s_3+2s_6) \mod 2n.$$ 
%$s-r = k_1+2k_2+2k_3+k_4+k_5+2k_6 = 0 \mod 2n$, and in particular $s-r$ is even, and thus 
Hence $(k_1+k_4+k_5)n=0 \mod 2n$, so $k_1+k_4+k_5$ is even. One of $k_1,k_4$, and $k_5$ must  be even, say $k_1=2t$ is even (the other two cases can be shown analogously). 
Then $r_1 b = (s_1 + k_1n)b =  (s_1 + 2nt)b = s_1b$ 
since $2 n b =0$ and we are done.

Extend $\alpha$ by setting $\alpha(b) := b$ 
and $$\alpha(a_1,a_2,a_3,a_4,a_5,a_6) := (a_2,a_1,a_4,a_3,a_6,a_5);$$ again, 
this determines $\alpha$ on all of $G$. 
Then $\alpha$ witnesses that $f$ is a pseudo-Siggers polymorphism: for each $r \in \omega$ we have
\begin{align*}
\alpha f(rb,1,rb,1,1,1) & = \alpha(a_1 + rb + a_3 +2rb) = \alpha(a_1 + a_3 + 3rb) \\
& = a_2 + a_4 + 3rb = a_2 + 2rb + a_4 + rb
= f(1,rb,1,rb,1,1)
\end{align*}
and it follows that $\alpha f_{\{1,3\}} = f_{\{2,4\}}$. 
Similarly, we get $\alpha f_{\{2,5\}} = f_{\{2,5\}}$
and $\alpha f_{\{4,6\}} = f_{\{4,6\}}$,
and hence $f$ is pseudo-Siggers by 
Proposition~\ref{prop:monoid-pseudo-sig}. 

%for all $x,y,z \in 
%Take self-embeddings $e_1,e_2$ of 
%$\bigoplus_{i \in \omega \setminus \{1,\dots,6\}} \langle a_i \rangle$ such that 
%$$e_1(g(x,y,x,z,y,z) = e_2(g(y,x,z,x,z,y)).$$

Finally, the implication from $(3)$ to $(1)$ is a well-known general fact for $\omega$-categorical structures that follows from the lift lemma presented in~\cite{canonical} (see~\cite{BartoPinskerDichotomy}): if the model-complete core of an $\omega$-categorial structure $\bB$ has a pseudo-Siggers polymorphism, then so has $\bB$. 
\end{proof} 

Item (3) in Theorem~\ref{thm:group-pseudo-sig}
is the condition from the first infinite-domain tractability conjecture (see~\cite{BartoPinskerDichotomy}). 
We already know from Theorem~\ref{thm:abelian-groups}, Theorem~\ref{thm:group-pseudo-sig},  
and from Theorem~\ref{thm:uch1} 
that if $\bG$ is an $\omega$-categorical
abelian group such that $(\bG,\neq)$ has a pseudo-Siggers polymorphism, then 
$\Pol(\bG,\neq)$ cannot have a uniformly
continuous minor-preserving map to $\Proj$,
unless P=NP. It is surprisingly difficult to verify
this also without the complexity-theoretic assumption; however, by bounding
the orbit growth\footnote{We thank Michael Kompatscher for discussing the growth rate of ${\mathbb Z}^{(\omega)}_2$.} of $(\bG,\neq)$ this follows from
a result of~\cite{BKOPP}, as we will see below. 

\begin{proposition}\label{prop:abelian-group}
Let $\bG$ be an $\omega$-categorical abelian group. Then $(\bG,\neq)$  
has 
a pseudo-Siggers polymorphism
if and only if $\Pol(\bG,\neq)$ has no
uniformly continuous minor-preserving map
to $\Proj$. 
\end{proposition}
\begin{proof}
Let $\bH$ be the model companion of $\bG$. 
First suppose that $(\bG,\neq)$ has no
uniformly continuous minor-preserving map
to $\Proj$. 
Then neither has $(\bH,\neq)$, by Proposition~\ref{prop:coreh1},
and so $(\bH,\neq)$ has a pseudo-Siggers polymorphism by Theorem~\ref{thm:BP}.
Hence, $(\bG,\neq)$ has a pseudo-Siggers polymorphism by Theorem~\ref{thm:group-pseudo-sig}. 

Now suppose that $(\bG,\neq)$ has a
uniformly continuous minor-preserving map
to $\Proj$.
It is known that if $\bB$ is an $\omega$-categorical structure with a uniformly continuous  minor-preserving map to $\Proj$ and if
the number of orbits
of $n$-tuples of $\Aut(\bB)$ 
grows slower than doubly exponentially,
then $\bB$ cannot have a pseudo-Siggers polymorphism~\cite{BKOPP-equations}. 
Steitz~\cite{Steitz} %(Theorem 1) 
proved that
for every $\omega$-categorical $\omega$-stable
structure $\bB$ there exists $m \in \omega$ such that
% of https://link.springer.com/content/pdf/10.1007%2FBF02773695.pdf): 
%Let T be w-categorical and w-stable. Then there is a natural number m such that

the number of orbits of $n$-tuples of $\Aut(\bB)$
is smaller than $2^{mn^2}$.
The groups 
${\mathbb Z}_n^{(\omega)}$ 
and ${\mathbb Z}_n^{(\omega)} \oplus {\mathbb Z}_{2n}$ 
are  for every $n \in \omega$ totally categorical:
the models of cardinality $\kappa$ are clearly  isomorphic to ${\mathbb Z}_n^{(\kappa)}$ 
or to ${\mathbb Z}_n^{(\kappa)} \oplus {\mathbb Z}_{2n}$, respectively, 
and hence in particular $\omega$-stable  (see~\cite{Tent-Ziegler}). 
\end{proof}

\section{Semilattices} 
\label{sect:semilattices}
A \emph{semilattice} is an algebra $(S;\wedge)$ where 
$\wedge$ is a binary operation that is associative, commutative, and idempotent. 
For $a,b \in A$ we define $a \leq b$ iff $a = ab = ba$. Clearly, $(A;\leq)$ is a partial order and polymorphisms of $(A;\wedge)$ 
are monotone with respect to $\leq$. 
Note that semilattices with a greatest element 
$\perp$ are special monoids (where the greatest element takes the role of $1$). 
As in the case of monoids, 
we often omit the symbol $\wedge$ and write $ab$ instead of $a \wedge b$. 

\begin{example}\label{expl:sn}
Let $\bP_n$ be the Boolean algebra with the atoms $\{1,\dots,n\}$. Clearly, 
the $\{\wedge\}$-reduct $\bS_n$ of $\bP_n$ is a semilattice. It is well-known that every finite semilattice 
$(S;\wedge)$ has the following embedding $e$ into $\bS_n$, for 
$n := |S|$, which we recall here for the convenience of the reader: 
if $b$ is any bijection between $S$ and $\{1,\dots,n\}$, define
$$e(x) := \bigvee_{y \leq x} b(y).$$
This map is injective:
if $e(x_1) \leq e(x_2)$
then in particular 
$b(x_1) \leq e(x_2)$
and since $b(x_1)$ is an atom 
we must have 
$b(x_1) \leq b(y)$ for some
$y \leq x_2$. 
Since $b(y)$ is an atom we must
have $x_1 = y$ and hence
$x_1 \leq x_2$. Together with the symmetric argument we obtain
that $e(x_1) = e(x_2)$ 
implies that $x_1=x_2$. 

Moreover, $e$ preserves $\wedge$:  let $x_1,x_2,x_3 \in S$ be such that $x_1 \wedge x_2 = x_3$.
Let $a$ be an atom of $\bS_n$
and $c \in S$ be such that $b(c) = a$. 
Then 
\begin{align*}
& a \leq e(x_1) \wedge e(x_2) \\
\text{ if and only if } \quad & 
a \leq e(x_1) \text { and } a \leq e(x_2) \\
\text{ if and only if } \quad & c \leq x_1 \text{ and } c \leq x_2 \\
\text{ if and only if } \quad & c \leq (x_1 \wedge x_2) = x_3 \\
\text{ if and only if } \quad & a \leq e(x_3). 
\end{align*}
%(see, for example, TODO1: reference or proof).
This shows that
$e(x_1) \wedge e(x_2) = e(x_3)$ and
concludes the proof. 
\demo
\end{example}

\begin{example}\label{expl:u}
The class $\mathcal L$ of all finite semilattices 
forms an amalgamation class. To see this, 
suppose that $\bB_1$ and $\bB_2$ are
two finite semilattices such that $B_1 \cap B_2$ 
is the domain of a subsemilattice $\bA$ of both
$\bB_1$ and $\bB_2$. We have to prove that there exists a finite semilattice $\bC$ and embeddings $e_1 \colon \bB_1 \hookrightarrow \bC$ and $e_2 \colon \bB_2 \hookrightarrow \bC$
such that $e_1(a) = e_2(a)$ for all $a \in A$. 
Note that the poset $(A;\leq)$ induced by 
$\bA$ is a subposet of the posets $(B_1;\leq)$
and $(B_2;\leq)$ induced by $\bB_1$ and $\bB_2$, respectively. Let $\bC$ be the 
Dedekind-McNeille completion~\cite{MacNeille} of 
the poset amalgam of $(B_1;\leq)$
and $(B_2;\leq)$. Then $\bC$ has the required properties. 

%Let $e$ be the embedding of $\bB_1$ into
%${\mathbb P}_{|B_1|}$ from Example~\ref{expl:sn}. 

We write 
$\bU$ for the Fra\"{i}ss\'e-limit of $\mathcal L$, i.e., for the up to isomorphism unique countable universal homogeneous semilattice, studied e.g.~in~\cite{DrosteKuskeTruss} where a finite axiomatisation of its first-order theory is presented. 
Note that $\bU$ is 
 $\omega$-categorical, because it is uniformly locally finite: in the subalgebra of $\bU$ 
 generated by 
 $u_1,\dots,u_n$ there are precisely
 the elements of the form $\bigwedge_{u \in V}$
 for a subset $V$ of $\{u_1,\dots,u_n\}$,
and hence their number is bounded by $2^n$. 
Also note that $\mathcal L$ is closed under taking finite direct products, and it follows from Proposition~\ref{prop:convex} 
that $\bU^2 \hookrightarrow \bU$.
% and $\bU$ is convex. 
\demo 
\end{example} 

$\Csp(\bU,\neq)$ might be viewed as a special case of Horn-Horn set constraints~\cite{BodHils}, which might be solved in polynomial time. 
We do not want to introduce Horn-Horn sets constraints here, but for the convenience show
how to derive a polynomial-time algorithm for
$\Csp(\bU,\neq)$ from 
Fact 24 in~\cite{BodHils}, which implies the following. 

\begin{lemma}\label{lemma:U}  Let $\phi=x\neq y \wedge \psi$ by a primitive positive formula of $(\bU,\neq)$ where $\psi$ does not contain formulas involving $\neq$. Then $\phi$ is satisfiable if and only if either  $\psi \wedge x=0 \wedge y=1$ or $\psi \wedge x=1 \wedge y=0$
is satisfiable in $(\{0,1\};\wedge,0,1)$.
\end{lemma} 
%We prove that 
%$\Csp(\bU,\neq)$ can be solved in polynomial time by reducing the problem to the satisfiability problem for Horn-Horn set constraints~\cite{BodHils}.
%
%\begin{definition}
%A quantifier-free formula in conjunctive normal
%form over the language of Boolean algebras is called \emph{Horn-Horn} if each clause
%contains at most one positive literal, and if each 
%positive literal is of the form 
%$\neg \overline{x_1} \vee \cdots \vee \overline{x_n} \vee y = 1$. 
%\end{definition}
%
%\begin{theorem}[Theorem 26 in~\cite{BodHils}]
%There exists a polynomial-time algorithm
%that decides whether a given Horn-Horn formula
%is satisfiable in a Boolean algebra. 
%\end{theorem}

\begin{proposition}\label{prop:alg} 
$\Csp(\bU,\neq)$ can be solved in polynomial time. 
\end{proposition}
\begin{proof}
Since $\bU^2 \hookrightarrow \bU$ (Proposition~\ref{prop:convex}) 
it suffices to consider the situation that $\phi$ is of the form $x \neq y \wedge \psi$ where $\psi$ does not contain
formulas involving $\neq$.  By the previous lemma  
$\phi$ is satisfiable if and only if either $\psi \wedge x=0 \wedge y=1$ or $\psi \wedge x=1 \wedge y=0$
is satisfiable in $(\{0,1\};\wedge,0,1)$. However $\Csp(\{0,1\};\wedge,0,1)$ can be solved in polynomial time (see, e.g., \cite{KlimaTessonTherien}), to which the result follows. 
%The algorithm in~\cite{BodHils} is formulated for so-called \emph{Horn-Horn set constraints}, and 
%An instance $\phi$ of $\Csp(\bU,\neq)$ can be translated to a Horn-Horn formula as follows. 
%\begin{itemize}
%\item We first replace each constraint of the form $(x_1 \wedge \cdots \wedge x_n) \neq (y_1 \wedge \cdots \wedge y_m)$
%by the three constraints $(x_1 \wedge \cdots \wedge x_n) = z_1$, $(y_1 \wedge \cdots \wedge y_m) = z_2$, and $z_1 \neq z_2$ where $z_1$ and $z_2$ are fresh variables. 
%\item Then each constraint of the form $(x_1 \wedge \cdots \wedge x_n) = (y_1 \wedge \cdots \wedge y_m)$
%is translated to $n+m$ Horn-Horn formulas
%\begin{align*}
%\overline{x_1} \vee \cdots \vee \overline{x_n} \vee y_1 = 1,
%\dots, \overline{x_1} \vee \cdots \vee \overline{x_n} \vee y_m = 1, \\
%\overline{y_1} \vee \cdots \vee \overline{y_m} \vee x_1 = 1, 
%\dots, 
%\overline{y_1} \vee \cdots \vee \overline{y_m} \vee x_n = 1. 
%\end{align*}
%\end{itemize} 
%The argument in Example~\ref{expl:sn} then shows that $\phi$ is satisfiable in $(\bU,\neq)$ if and only if the translation is satisfiable in $(\bP_n,\neq)$ for some $n \in \omega$. 
\end{proof}

In this section we prove the following dichotomy result.

\begin{theorem}\label{thm:semilattice}
Let %$(S;\wedge)$ 
$\bS$ 
be a countable $\omega$-categorical semilattice.  Then either 
\begin{enumerate}
\item there is a uniformly continuous
minor-preserving from $\Pol(\bS,\neq)$ to the clone of projections, in which case $\Csp(\bS,\neq)$ is NP-hard, or
\item the model-companion $\bC$ of $\bS$ is isomorphic to $\bU$. 
In this case $\Pol(\bS,\neq)$ is in P.
\end{enumerate}
\end{theorem}

We first prove that the two cases are indeed disjoint.
The identities that appear in the next proposition have been discovered by Jakub Rydval in a different context~\cite{RydvalFP};
note that these identities have \emph{height one} and hence are preserved by all minor-preserving maps~\cite{wonderland}. 

\begin{proposition}\label{prop:jakub}
There are $f,g_1,\dots,g_4 \in \Pol(\bU,\neq)$ such that for all $x,y \in U$ 
 	\[ \begin{array}{c} g_{1}(y , x , x) =   f(x,y,x,x),   \\
  	g_{1}(x , y , x) =  f(x,x,y,x),  \\
  	g_{1}(x , x , y) =   f(x,x,x,y),
  	\end{array}  \ 
  	\begin{array}{c} g_{2}(y , x , x) =   f(y,x,x,x),   \\
  	g_{2}(x , y , x) =  f(x,x,y,x),  \\
  	g_{2}(x , x , y) =   f(x,x,x,y), \end{array}     \] 
  	\[ \begin{array}{c} g_{3}(y , x , x) =   f(y,x,x,x),   \\
  	g_{3}(x , y , x) =  f(x,y,x,x),  \\
  	g_{3}(x , x , y) =   f(x,x,x,y),
  	\end{array}  \ 
  	\begin{array}{c} g_{4}(y , x , x) =   f(y,x,x,x),  \\
  	g_{4}(x , y , x) =  f(x,y,x,x),  \\
  	g_{4}(x , x , y) =   f(x,x,y,x). \end{array}  \]   
\end{proposition}
\begin{proof}
Let $\Phi$ be the following (infinite)
set of atomic $\{\wedge,\neq\}$-formulas. The variables of these formulas consist
of the elements of semilattices 
$(W_0;\wedge)$, $(W_1;\wedge)$,
%(W_2,\wedge), (W_3,\wedge), 
\dots, $(W_4;\wedge)$ 
such that there is an isomorphism $\alpha_0$ from $\bU^4$ to $(W_0;\wedge)$ and an isomorphism $\alpha_i$ 
from $\bU^3$ to 
 $(W_i;\wedge)$ for each 
$i \in \{1,\dots,4\}$. 
The formulas in $\Phi$ come from three groups.
\begin{itemize}
\item For $i \in \{0,\dots,4\}$ and $x,y,z \in W_i$ the set $\Phi$ contains
the atomic formula $(x \wedge y) = z$ if and only if $(x \wedge z) = z$ holds in $(W_i;\wedge)$. 
\item For $i \in \{0,\dots,4\}$ and distinct $x_1,x_2 \in W_i$ the set $\Phi$ contains $x_1 \neq x_2$. 
\item Finally, whenever $g_i(x_1,x_2,x_3) = f(y_1,y_2,y_3,y_4)$ is an identity from
the statement then the set $\Phi$ contains
the atomic formula $x=y$
where $x$ is the variable
$\alpha_i(x_1,x_2,x_3) \in W_i$
and $y$ is the variable $\alpha_0(y_1,\dots,y_4) \in W_0$. 
\end{itemize}
Note that any satisfying assignment of $\Phi$ in 
$\bU$ restricted to the elements of $W_i$ defines a polymorphism of $\bU$ because of the atomic formulas 
of the first group; these polymorphisms will be injective because of the atomic formulas of the second group; and
jointly they satisfy the identities given in the statement because of the atomic formulas of the third group. 

Since $\bU$ is $\omega$-categorical, it suffices to show that every finite 
subset $\phi$ of formulas in $\Phi$ is satisfiable. If $\phi$ contains no
atomic formulas of the form $x \neq y$ then $\phi$ is trivially satisfiable by mapping all variables to the same element of $\bU$. 
Since $\bU^2 \hookrightarrow \bU$ 
it suffices by Proposition \ref{prop:convex} to consider the situation that $\phi$ is of the form $x \neq y \wedge \psi$ where $\psi$ does not contain
formulas involving $\neq$. By Lemma \ref{lemma:U} it suffices to show that either
$\psi \wedge x=0 \wedge y=1$ or 
or $\psi \wedge x=1 \wedge y=0$
is satisfiable in $(\{0,1\};\wedge,0,1)$. 
We construct a solution to 
$\psi \wedge x=1 \wedge y=0$. 
Suppose that 
$x$, which must be assigned to $1$, is of the form
$\alpha_1(y,x,x)$. 
If the variable $x_0 := \alpha_0(x,y,x,x)$ is 
present in $\phi$, assign it to $1$, too. 
If the variable $x_3 := \alpha_3(x,y,x)$ is present in $\phi$, assign it to $1$, too. 
If the variable $x_4 := \alpha_4(x,x,y)$ is present in 
$\phi$, assign it to $1$, too. 
Note that assigning all other variables to $0$ is a satisfying solution to $\phi$. 
\end{proof}

Before we prove Theorem~\ref{thm:semilattice} we consider a special case that will be used in the proof. Every linear order $(A;\leq)$ 
%(and more generally every partial order with unique greatest lower bounds) 
gives rise to a semilattice $(A;\wedge)$, by defining $a \wedge b := a$ iff $a \leq b$. Semilattices that arise from linear orders in this way are called \emph{linear}. 
A semilattice $(S;\wedge)$ is called \emph{semilinear} if for every $x \in S$ the subsemilattice induced by $\{y \in S \mid y \leq x\}$ is linear. 

%\begin{lemma}\label{lem:lin}
%Let $(A;\wedge)$ be the semilattice of a countable $\omega$-categorical linear order, and let  $(C;\wedge)$ be its model companion. 
%Then $(C;\wedge,\neq)$ does not have a pseudo-Siggers polymorphism (and hence $\Csp(A;\wedge,\neq)$ is NP-hard). 
%In fact, the polymorphism clone of an expansion of the model companion of
%$(A;\wedge)$ by finitely many constants and $\neq$ 
%has a continuous homomorphism to $\Proj$. 
%\end{lemma}

\begin{lemma}\label{lem:semi-lin}
Let $(S;\wedge)$ be a countable  
$\omega$-categorical semilinear 
semilattice with $|S| > 1$.
% at least two elements. 
%with a linear model-companion $(C;\wedge)$. 
%Then 
%\begin{itemize}
%\item $(C;\wedge,\neq)$ does not have a pseudo-Siggers polymorphism, 
%\item 
Then 
$\Pol(S;\wedge,\neq)$
has a uniformly continuous minor-preserving map to $\Proj$.
%, and \item $\Csp(S;\wedge,\neq)$ is NP-hard. 
%\end{itemize}
\end{lemma}

\begin{proof}
If $|S|=2$ then the
statement follows
from the proof of Schaefers classification:  $(S;\wedge,\neq)$
is neither preserved by $\min$, $\max$, or constant operations because these operations do not preserve $\neq$, and neither preserved by $\minority$ and 
$\majority$ because these operations
do not preserve $\wedge$. 
If $|S| > 2$ is finite
then the statement holds
because of the well-known fact that 
in this case all polymorphisms
of $(S;\neq)$ only depend on one argument.  
Otherwise, $(S;\wedge)$ is countably infinite, and so is its model companion $(C;\wedge)$. 
First consider the case that 
$(C;\wedge)$ is linear. 
Every $\omega$-categorical infinite linear order embeds all finite linear orders, and thus $(C;\wedge)$ and $({\mathbb Q};\min)$ have the same age. Since both structures are model-complete, they have the same first-order theory, and 
by the $\omega$-categoricity of 
$({\mathbb Q};\min)$ they are isomorphic. 
%NP-hardness of
%$\Csp({\mathbb Q};\min,\neq)$ 
% follows from the classification result given in~\cite{tcsps-journal}. 
 %For the convenience of the reader we show how to prove NP-hardness by  reduction from $({\mathbb Q};S)$ where 
 Let $$R := \{(x,y,z) \in {\mathbb Q}^3 \mid x=y < z \vee x = z <  y\}.$$
 %, which is NP-hard by Proposition~15 in~\cite{tcsps-journal}. 
%We claim that the 
The primitive positive formula $\psi(x,y,z)$ given by
$$\min(y,z) = x \wedge y \neq z$$
defines $R$ in $({\mathbb Q};\min,\neq)$.
%, so the NP-hardness of $({\mathbb Q};\min,\neq)$ follows from Lemma~\ref{lem:pp-reduce}. 
%Moreover,
%Since every finite structure is homomorphically equivalent to a structure with a primitive positive interpretation in 
It follows from Theorem 51 in~\cite{BP-reductsRamsey} in combination with Theorem~28  in~\cite{Topo-Birk} that 
$\Pol({\mathbb Q};R)$ has a (uniformly) continuous homomorphism to $\Proj$ 
and hence the same holds for $\Pol({\mathbb Q};\min,\neq)$. 
It follows in particular that there exists a uniformly
continuous minor-preserving map from $\Pol({\mathbb Q};\min,\neq)$ to $\Proj$ (see~\cite{wonderland}). 

To prove the general case, 
choose an element $x \in C$ such that $T := \{y \in C \mid y \leq x\}$ contains more than one element; such an element $x$ must exist because $C$
has more than one element. 
Then the subsemilattice of $(C;\wedge)$ with domain $T$ is linear,
and we have proved above 
that $\Pol(T;\wedge,\neq)$ has a 
uniformly continuous minor-preserving map to $\Proj$. Since $T$ is primitive positive definable in an expansion of the model-complete core $(C;\wedge,\neq)$, it follows that
$\Pol(C;\wedge,\neq)$ has a uniformly continuous minor-preserving map to $\Proj$, too (Proposition~\ref{prop:coreh1}). 
\end{proof} 

\begin{proof}[Proof of Theorem~\ref{thm:semilattice}]
If there is a uniformly continuous minor-preserving map from $\Pol(\bS,\neq)$ to $\Proj$, then the NP-hardness of
$\Csp(\bS,\neq)$ follows from Theorem~\ref{thm:uch1}. 
In this case, the model-companion of $\bS$
cannot be isomorphic to $\bU$:
otherwise, there would be 
a minor-preserving map from $\bU$
via $\Pol(\bS,\neq)$ to $\Proj$,
and such a map preserves the identities from Proposition~\ref{prop:jakub} (see~\cite{wonderland}). Clearly,
these identities cannot be satisfied by projections on a two-element set,
so we reached a contradiction. 

Now suppose that there
is no uniformly continuous minor-preserving map from $\Pol(\bS,\neq)$ to $\Proj$. 
We will prove that for every $n \in \omega$, 
$\bS_n$ embeds into every infinite 
$\omega$-categorical 
semilattice $\bT$ 
such that $\Pol(\bT,\neq)$ has no uniformly continuous minor-preserving map to $\Proj$.  
The proof is by induction on $n \in \omega$. The statement is clear for $n=1$, 
and for $n=2$ the statement follows from
Lemma~\ref{lem:semi-lin}. 
% and hence also into $(S;\wedge)$. 
%This shows the induction base case for $n=2$. 
Now suppose inductively that $n \geq 2$ 
is such that 
$\bS_n$ embeds into every
$\omega$-categorical 
semilattice $\bT$ 
such that $\Pol(\bT,\neq)$ has no uniformly continuous minor-preserving map to $\Proj$. 
We would like to show that
$\bS_{n+1}$ embeds into $\bT$, too.  We may assume that 
$\bT$ is model-complete;
otherwise, 
let $\bC$ be the model-companion of $\bT$, 
which exists and is again an $\omega$-categorical 
countable semilattice (Lemma~\ref{lem:mc}). 
Moreover, there is a uniformly continuous minor-preserving map
from $\Pol(\bT,\neq)$ to
$\Pol(\bC,\neq)$, and
hence there is no uniformly continuous minor-preserving map
from $\Pol(\bC,\neq)$ to
$\Proj$. We may therefore replace 
$\bT$ by $\bC$. 

\medskip 
{\bf Claim 1.} For all $a,b \in T$
with $a<b$ the subsemilattice
of $\bT$ induced by the interval 
$[a,b] := \{x \in T \mid a \leq x \leq b\}$ embeds $\bS_n$. Clearly, the structure $(\bT,\neq,a,b)$ is a model-complete core, too. 
Proposition~\ref{prop:coreh1} thus implies
that $(\bT,\neq,a,b)$ has no
uniformly continuous minor-preserving map to 
$\Proj$. 
The set $[a,b]$ is primitive positive definable in 
$(\bT,a,b)$, and $([a,b];\wedge)$ is an $\omega$-categorical subsemilattice of $\bT$. 
%The restriction of $s$
%to $[a,b]$ is a pseudo-Siggers polymorphism of 
%$([a,b];\wedge,\neq)$.
%Note that $a$ is the least element in $(\upa;\wedge)$ and $a x = x a = x$ for all $x \in \upa$, which makes
%$([a,b];\wedge)$ a monoid; 
%we thus also write $1$ instead of $a$. 
So there is a uniformly continuous minor-preserving map from $\Pol(\bT,\neq,a,b)$ to 
$\Pol([a,b];\wedge,\neq)$ and hence 
 $\Pol([a,b];\wedge,\neq)$ cannot have  
a uniformly continuous minor-preserving map to $\Proj$. 
 %Since $\Pol(\upa,\wedge,\neq)$  does not have a uniformly continuous minor-preserving map to $\Proj$, 
The inductive assumption then implies that there exists an embedding $e \colon \bS_n \hookrightarrow ([a,b];\wedge)$.

\medskip 
{\bf Claim 2.} There exist $a,b,c \in T$ 
such that $a<b<c$ and
\begin{align}
f(b,a,b,a,a,a) & \neq f(c,a,c,a,a,a) \label{eq:1}  \\
\text{ or } f(a,b,a,a,b,a) & \neq f(a,c,a,a,c,a). \label{eq:2} 
% \\
%\text{ or } f(a,a,a,b,a,b) & \neq f(a,a,a,c,a,c).  
\end{align}
Suppose otherwise. 
By Claim 1, there exists an embedding 
$e \colon \bS_2 \to \bT$. \\
% (EXPAND), which shows that there is an embedding $e \colon \bS_2 \to (T;\wedge)$.
Let $y,z \in T$ be
the images of the atoms of $\bS_2$
under $e$, and let $w := e({\bf 1})$ and $x := e({\bf 0})$. \\
Again by Claim 1, 
there exists an embedding $e'$  
of $\bS_2$ into $[y,w]$. \\
Let $u$ and $v$ be the images of the atoms of $\bS_2$ under $e'$; \\
we may then replace $y$ by $uv$ and then 
$x$ by $yz$; \\
in the resulting constellation
we have $uv=y$ and $yz=x$; \\
see the figure to the right.
%~\ref{fig:semilattice}.

\vspace{-2.6cm}
\begin{flushright}
\includegraphics[scale=.4]{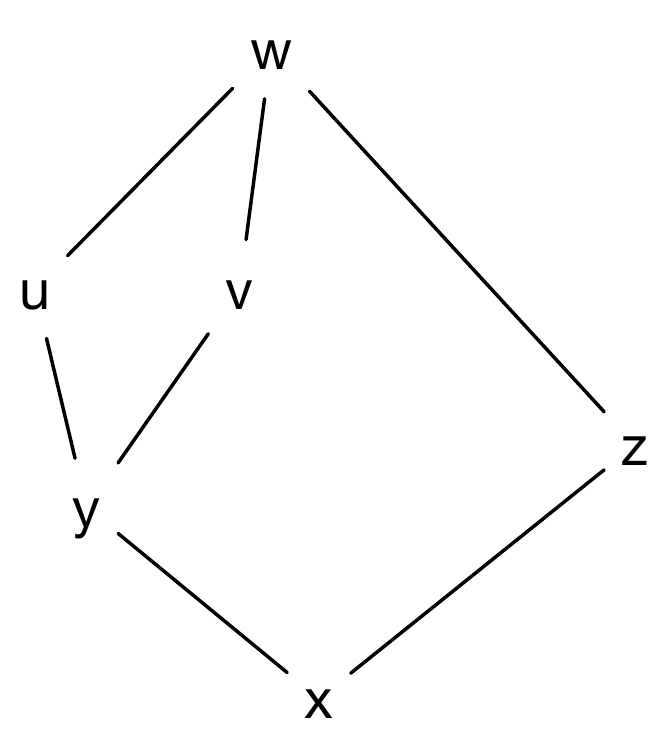}
\end{flushright}

Since we assumed that (\ref{eq:1}) neither holds for $(a,b,c) = (x,y,w)$ nor for $(a,b,c) = (x,z,w)$, 
we have 
$$f(y,x,y,x,x,x) = f(w,x,w,x,x,x) = f(z,x,z,x,x,x)$$
and so 
$$f(x,x,x,x,x,x) = f(y,x,y,x,x,x) f(z,x,z,x,x,x) = f(y,x,y,x,x,x).$$
A similar argument applied to the image of $e'$ 
and (\ref{eq:2}) shows that $f(y,y,y,y,y,y) = f(y,w,y,y,w,y)$. Hence, 
\begin{align*}
f(y,z,y,x,z,x) & = f(y,w,y,y,w,y) f(y,z,y,z,z,z) \\
& = f(y,y,y,y,y,y) f(y,z,y,z,z,z) \\
& = f(y,x,y,x,x,x) = f(x,x,x,x,x,x)
\end{align*}
and so by (\ref{eq:siggers}) 
$$f(y,z,y,x,z,x) = f(x,x,x,x,x,x) = f(z,y,x,y,x,z)$$
contradicting that $f$ preserves $\neq$. 

Suppose that $f(b,a,b,a,a,a) \neq f(c,a,c,a,a,a)$; the other case from Claim 2 can be shown similarly. Let $g \colon \{b,c\} \times [a,b] \to T$ be given by 
$$g(x,y) := f(x,y,x,y,y,y).$$ 
We show that $g$ is injective. Suppose first that $g(b,y) = g(b,y')$ for $y,y' \in [a,b]$.
Then by (\ref{eq:siggers})  we have 
$f(y,b,y,b,y,y) = f(y',b,y',b,y',y')$. 
Since $yb=y$ and $y'b=y'$ we have
$$f(y,y,y,y,y,y) = g(b,y)f(y,b,y,b,y,y) = g(b,y') f(y',b,y',b,y',y') = f(y',y',y',y',y',y')$$ 
and so $y=y'$ as $f$ preserves $\neq$. A similar argument shows that
$g(c,y) = g(c,y')$ forces $y=y'$. 
Finally, if $g(b,x) = g(c,x)$ then multiplying both sides by
$f(c,a,c,a,a,a)$ we obtain 
$f(b,a,b,a,a,a) = f(c,a,c,a,a,a)$, contradicting our hypothesis. 
This shows that $\bS_1 \times \bS_n$ embeds into $\bT$. 
Since $\bS_1 \times \bS_{n}  \simeq \bS_{n+1}$ we have that  
$\bS_{n+1}$ embeds into $\bT$, which concludes the induction. 

We obtain that $\bS$
and $\bC$ have the same age as $\bU$. Since $\bC$ and
$\bU$ are model-complete, they must be isomorphic. The polynomial-time tractability of $\Csp(\bS,\neq)$
then follows from Proposition~\ref{prop:alg}. 
\end{proof}

\begin{corollary}
Let $\bS$ be a countable $\omega$-categorical semilattice. Then $\Csp(\bS,\neq)$ is in P or NP-hard. 
\end{corollary}
%\begin{proof}
%Let $(C;\wedge)$ be the model companion of $(A;\wedge)$, so that $(C;\wedge,\neq)$ is a model-complete core. If $(C;\wedge,\neq)$ has a pseudo-Siggers polymorphism, the statement follows from
%Theorem~\ref{thm:semilattice}. Otherwise, 
%the statement follows from Theorem~\ref{thm:BP}. 
%\end{proof}

\section{Conclusion and Future Work}
\label{sect:discussion}
In previous work about 
CSPs for $\omega$-categorical structures 
$\bA$, algebras were used to analyse the polymorphism clones of $\bA$ (and these algebras are oligomorphic, but never $\omega$-categorical). 
In this article, in contrast, the structure $\bA$ 
itself is assumed to be an algebra, expanded by the disequality relation. 
%Note that $\omega$-categorical semilattices,  and abelian groups are typically \emph{not} first-order reducts of homogeneous structures with finite relational signature
%(this is fairly easy to prove for example for ${\mathbb Z}_p^\omega$ or for the universal homogeneous semilattice $\bU$), 
%so in particular not 
%first-order reducts 
%of finitely bounded homogeneous structures. 
%Hence,
Our result underlines the importance of 
\begin{itemize}
\item uniformly continuous minor-preserving maps 
$\Pol(\bA,\neq) \to \Proj$ as a tool for proving hardness (in Section~\ref{sect:semilattices}),
and 
\item pseudo-Siggers polymorphisms of model-complete cores to obtain structural (and subsequently algorithmic) results. 
\end{itemize} 
%outside the scope of the infinite-domain tractability conjecture. 
We are not aware of any previous result that would 
use the Siggers (or pseudo-Siggers) identity directly, even for algebras over finite domains. 
For example,
the result that
every structure with a finite domain that has a Siggers polymorphism also has a cyclic polymorphism
 departs from the (a priori) weaker assumption that the structure has a Taylor polymorphism~\cite{Cyclic}. We close with some open problems. 
\begin{itemize}
%\item If $G$ is an $\omega$-categorical group
%with a Siggers polymorphism, but no binary pseudo-symmetric polymorphism, is $\Csp(G,\neq)$ polynomial-time equivalent to $\Csp(G/\langle x \rangle,\neq)$  where $x$ is the pp-definable involution of $G$?
% source?
%\item Does there exist a non-abelian $\omega$-categorical group $G$ such that $\Csp(G;\neq)$ can be solved in polynomial time? Tom has a candidate. 
%\item Does the converse of 
\item Let $\bA$ be an $\omega$-categorical model-complete algebra. Suppose that $\bA^2 \hookrightarrow \bA$. Does then $(\bA,\neq)$ have a binary pseudo-symmetric polymorphism? The converse is false (a counterexample
can be found in the class of $\omega$-categorical model-complete algebras with a single unary function symbol). 
 The forward implication is true if $\bA^2$ is isomorphic to $\bA$~\cite{maximal}. 
%\item Is there an $\omega$-categorical non-abelian group $\bA$ such that $\Csp(\bA;\neq)$ is in P? 
%\item Does the universal semi-lattice satisfy Jakub's identities? 
\item Let $\bL$ be an $\omega$-categorical lattice. Does $\Csp(\bL,\neq)$ satisfy
a complexity dichotomy P versus NP-hard? We may assume that 
$\bL$ is model-complete. 
Using similar techniques as in Section~\ref{sect:semilattices} it is possible to show that if 
$\Pol(\bL,\neq)$ does not have uniformly continuous minor-preserving maps to $\Proj$, then $\bL$ must have a square embedding $\bL^2 \hookrightarrow \bL$. So we may assume that 
the age of $\bL$ is closed under taking products 
and thus contains the class of all distributive lattices. 
The countable homogeneous universal distributive lattice has
a uniformly continuous continuous minor-preserving maps to $\Proj$ (this has essentially been observed in~\cite{BodHils}). 
Are there any other examples of model-complete $\omega$-categorical lattices 
with $\bL^2 \hookrightarrow \bL$? 
%Component Subsets of the Free Lattice on n Generators
%Author(s): Richard A. Dean
%Proceedings of the American Mathematical Society, Vol. 7, No. 2 (Apr., 1956), pp.220-226
Note that the class of all finite lattices forms
an amalgamation class, but the countable homogeneous lattice $\bL$ with this age is not $\omega$-categorical: this can be seen from the fact that every finite lattice can be embedded
in a finite lattice with three generators~\cite{Dean} (page 224). Hence, there exists an infinite
number of inequivalent formulas with three variables over $\bL$. 
\end{itemize}

\bibliographystyle{abbrv}
\bibliography{local}

\end{document}